\documentclass{article}
\usepackage[utf8]{inputenc}
\usepackage[margin=1in]{geometry}
\usepackage{graphicx, amssymb,amsmath,amsthm,xcolor,enumitem,mathtools}
\numberwithin{equation}{section}
\newtheorem{theorem}{Theorem}[section]
\newtheorem{corollary}[theorem]{Corollary}
\newtheorem{lemma}[theorem]{Lemma}

\theoremstyle{definition}
\newtheorem{example}[theorem]{Example}
\theoremstyle{remark}
\newtheorem{remark}{Remark}[section]
\usepackage[mathlines]{lineno}
\usepackage[
    colorlinks=true,   
    linkcolor=blue,    
    urlcolor=blue,     
    citecolor=blue     
]{hyperref}

\newcommand{\E}{\mathbf{E}}

\title{Spectral Properties of Elementwise-\\Transformed Spiked Matrices}
\author{Michael J. Feldman \\ Department of Statistics, Stanford University}

\begin{document}
\date{} 

\maketitle


\begin{abstract}
This work concerns elementwise transformations of spiked matrices: $Y_n = n^{-1/2} f( \sqrt{n} X_n + Z_n)$. Here, $f$  is a function applied elementwise, $X_n$ is a low-rank signal matrix,  and $Z_n$ is white noise. 
We find that principal component analysis (PCA)  is capable of recovering 
signal under highly nonlinear or discontinuous transformations. Specifically, in the high-dimensional setting where  $Y_n$ is of size $n \times p$ with $n,p \rightarrow \infty$ and $p/n \rightarrow \gamma > 0$, we uncover a phase transition: 
 for signal-to-noise ratios above a precise threshold---depending on $f$, the distribution of elements of $Z_n$, and the limiting aspect ratio $\gamma$---the principal components of $Y_n$  (partially) recover  those of $X_n$. Below this threshold, the principal components of $Y_n$ are asymptotically orthogonal to the signal. 
In contrast, in the standard setting where PCA is applied to  
$X_n + n^{-1/2}Z_n$ directly, the analogous phase transition  depends only on $\gamma$. Similar phenomena occur with $X_n$ square and symmetric and $Z_n$ a generalized Wigner matrix. 

This model accommodates diverse data types not covered by prior spiked-matrix theory,  including forms of discrete data, preprocessed data, and data with missing values. Our results provide theoretical justification for applying PCA to such data, helping to elucidate PCA's empirical success.


\end{abstract}

\section{Introduction} \label{sec1}

From principal component analysis  to covariance estimation to factor analysis, spiked matrices are widely used to model high-dimensional data with latent low-rank structure: 
\begin{align}
    Y_n \coloneqq X_n + \frac{1}{\sqrt{n}} Z_n \label{-1},
\end{align}
where $Y_n$ is the observed data, $X_n$ is a low-rank signal matrix, and $Z_n$ is a noise matrix, each of size $n \times p$. 

The spectral properties of this model in high dimensions are well understood; we point to a few key papers \cite{BBP, BGN11, BGN12, P07, Peche},  the surveys \cite{JP18, Paul}, and the references therein. In settings where $n$ and $p$ are comparable, the singular values and vectors (or principal components) of $Y_n$ are inconsistent estimators of those of $X_n$.
Assume the elements of $Z_n$ are independent and identically distributed (i.i.d.) with mean zero and finite fourth moment. As $n, p \rightarrow \infty$ with $p/n \rightarrow \gamma \in (0, \infty)$, a phase transition occurs: the principal components of $Y_n$ contain information about the signal $X_n$ if and only if the signal-to-noise ratio exceeds $\gamma^{1/4}$.  Below this threshold, the principal components are asymptotically orthogonal to the signal. Analogous results hold 
with $X_n$ square and symmetric and $Z_n$ a generalized Wigner matrix. This phase transition, in the context of  Johnstone's closely related spiked covariance model, is known as the Baik--Ben Arous--P\'ech\'e threshold. 


This work concerns the following  generalization of model (\ref{-1}):
\begin{align} \label{0} Y_n \coloneqq \frac{1}{\sqrt{n}} f( \sqrt{n} X_n + Z_n) , \end{align}
where $f: \mathbb{R} \rightarrow \mathbb{R}$ is applied elementwise. 

We find that PCA is capable of recovering low-rank signal under highly nonlinear or discontinuous transformations. More specifically, phenomena observed under model (\ref{-1}), described above, extend to the elementwise-transformed setting (\ref{0}). As $n, p \rightarrow \infty$ with $p/n \rightarrow \gamma \in (0, \infty)$, a phase transition occurs: 
 above a signal-to-noise ratio threshold---depending on $f$, the law $\mu$ of elements of $Z_n$, and the limiting aspect ratio $\gamma$---the principal components of  $Y_n = n^{-1/2} f(\sqrt{n} X_n + Z_n)$ (partially) recover  those of $X_n$. Below this threshold, the principal components are asymptotically orthogonal to the signal. This is in contrast to the standard setting (\ref{-1}), where the analogous phase transition depends only on $\gamma$ and the variance of noise. Similar phenomenona occur with $X_n$ square and symmetric and $Z_n$ a generalized Wigner matrix. 

Informally stated, our main result is the following: there exists a 
constant $\tau(f,\mu)$ 
such that certain spectral properties of $Y_n$ are asymptotically equivalent to those of 
\begin{align}\label{0.1}
\tau(f,\mu) X_n 
+  \frac{1}{\sqrt{n}}Z_n .
\end{align}
That is, in high dimensions, the principal components of $Y_n$ behave as those of a standard spiked matrix with signal term rescaled by $\tau(f,\mu)$. For $f$ and $\mu$ such that $\tau(f,\mu) = 0$, PCA fails for signal-to-noise ratios of order one. In this case, there exists an $f, \mu$-dependent integer $\ell_* \geq 2$ (assuming $f \neq 0$ $\mu$-almost everywhere) such that PCA is powerful provided  the signal-to-noise ratio of $Y_n$ scales as $n^{1-1/(2\ell_*)}$. 

As $f$ may be discontinuous and $\mu$ may not have a density function, analysis is delicate.
Our approach is based on expanding $f$ in a basis of orthogonal polynomials with respect to $\mu$---such tools are used in the study of  kernel matrices in high dimensions \cite{Cheng,  ElK, liao2020sparse, LY23, theo}. Of these prior works, ours is most similar to \cite{liao2020sparse}, which studies elementwise transformations of the Gram matrix under a Gaussian mixture model. Recent independent work \cite{Guionnet} considers model (\ref{0}) under related assumptions with quite similar results and proof techniques. An advantage of our work is that \cite{Guionnet} assumes $f$ is locally Lipschitz---discontinuous transformations are of both theoretical and practical interest (see (\ref{plmokn}) below). 

\subsection{Motivation and Applications} \label{sec:moti}
This work confronts a limitation of the current theory of PCA: while high-dimensional studies of PCA standardly assume the spiked model, PCA is often applied to diverse forms of data that (\ref{-1}) cannot accommodate, including discrete data, preprocessed data, and data with missing values. 
Model (\ref{0}) addresses these data types, and our results provide theoretical justification for applying PCA and help to elucidate its empirical success.
To precisely describe high-dimensional phenomena, we require certain technical conditions that readers might question the generality and verifiability of. While we argue (Remarks \ref{rem1}--\ref{rem210}) that these assumptions are reasonable, we believe 
the message of this work---that $X_n$ ``reappears" in the spectrum of $Y_n$, and that PCA can therefore recover signal---is much more general than the specific conditions we impose, though exact asymptotics may not be achievable.

To highlight the gap between the theory and practice of PCA, we reference in particular \cite{N08}, which examines a genetic dataset with several characteristics of spiked matrices: the sample covariance matrix of the data has two outlier eigenvalues that ``carry signal," and the eigenvalue histogram is well approximated by the Marchenko--Pastur law (the limiting spectral distribution (LSD) of $Y_n^\top Y_n$ under model (\ref{-1}) with white noise)---see Figure \ref{fig0}. Yet, the data in \cite{N08} takes values in $\{0,1,2\}$, which (\ref{-1}) cannot generate. We propose to model such data as
\begin{align}\label{-2}
y_{ij} \sim \text{Bin}(2, \text{logistic}(x_{ij})) ,
\end{align}
where $\text{logistic}(x) \coloneqq (1+e^{-x})^{-1}$ and $x_{ij}$ and $y_{ij}$ denote the elements of $X_n$ and $Y_n$, respectively.
A complete analysis of this model is given in Section \ref{sec:ex};  $Y_n$ is (approximately)  a transformation of a spiked matrix, and the spectrum of $n^{-1} Y_n^\top Y_n$ (1) converges to the Marchenko--Pastur law, and (2) contains outlier eigenvalues corresponding to $X_n$. These findings result from reformulating (\ref{-2}):  if $z_{ij}$ is logistically distributed, 
\begin{align} \label{plmokn}
{\bf 1}(-x_{ij} + z_{ij} \leq 0) \sim \text{Ber}(\text{logistic}(x_{ij})) 
\end{align}
(summing two i.i.d.\ copies of this model yields (\ref{-2})). 

Maximum likelihood estimation (subject to a low-rank constraint) is studied under (\ref{plmokn}) in \cite{lpca2, lpca1}; however, in genetic applications, practitioners seem to favor PCA over alternatives. To our knowledge, this is the first work to provide theoretical justification for applying PCA to data of this form.

Additional applications in Section \ref{sec:ex} include the following: 

\begin{enumerate}
    \item The ReLU activation, $f(z) = \max(z,0)$. This is a particular form of missing data: negative values are unobserved. Under Gaussian noise, the effect of $f(z)$ is to raise the recovery threshold of PCA by a factor of $\sqrt{2(\pi-1)/\pi}$.
    \item Truncated data, $f_c(z) = z {\bf 1}(|z| \leq c)$. Observed data may be inherently truncated, or truncation may be intentionally applied to the data as a preprocessing step.   The effect of truncation on PCA depends heavily on the distribution of noise---under Gaussian noise, truncation raises the recovery threshold, while under heavy-tailed noise, truncation may dramatically lower the recovery threshold. Our results enable calculation (for a given noise distribution) of the optimal truncation level. 
    
    For example, with Cauchy-distributed noise, this level is $c^*\approx 2.028$. Without truncation, the singular vectors of $X_n + n^{-1/2} Z_n$ are asymptotically orthogonal to those of $X_n$.
    \item Optimal elementwise preprocessing. For certain classes of noise distributions, we identify an optimal elementwise preprocessor that strictly lowers the recovery threshold of PCA, extending \cite{Perry, Mont}.   
    \item Under the spiked model, there is an optimal eigenvalue shrinkage function (under operator norm loss) for estimation of $X_n$ given $X_n + n^{-1/2} Z_n$ (see \cite{GD14, GD17, leeb, DF22}). We prove that this shrinkage function is optimal under model (\ref{0}) as well. 
    
\end{enumerate}

\begin{figure}[]
\centering
\includegraphics[height=2.2in]{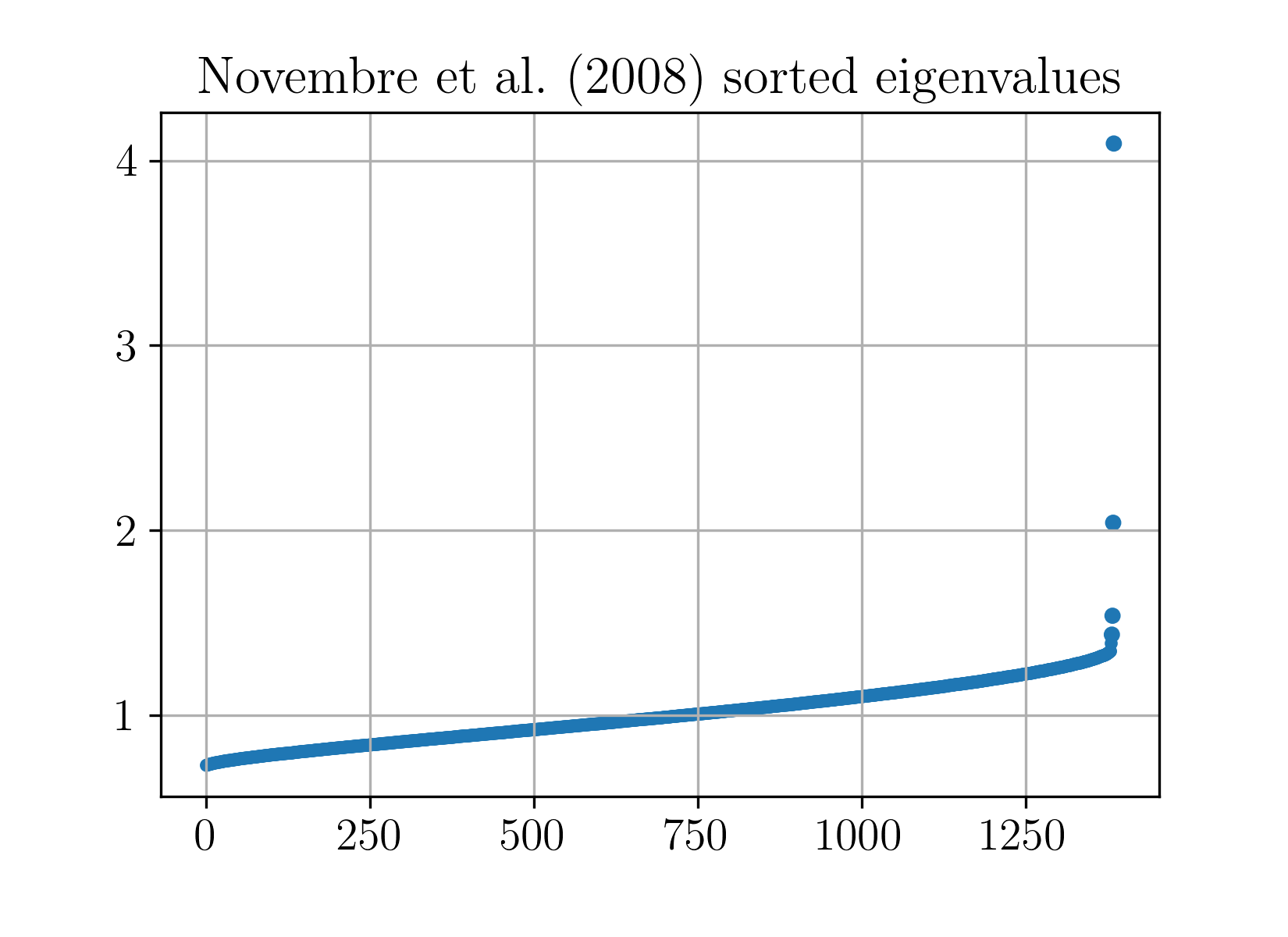}
\includegraphics[height=2.2in]{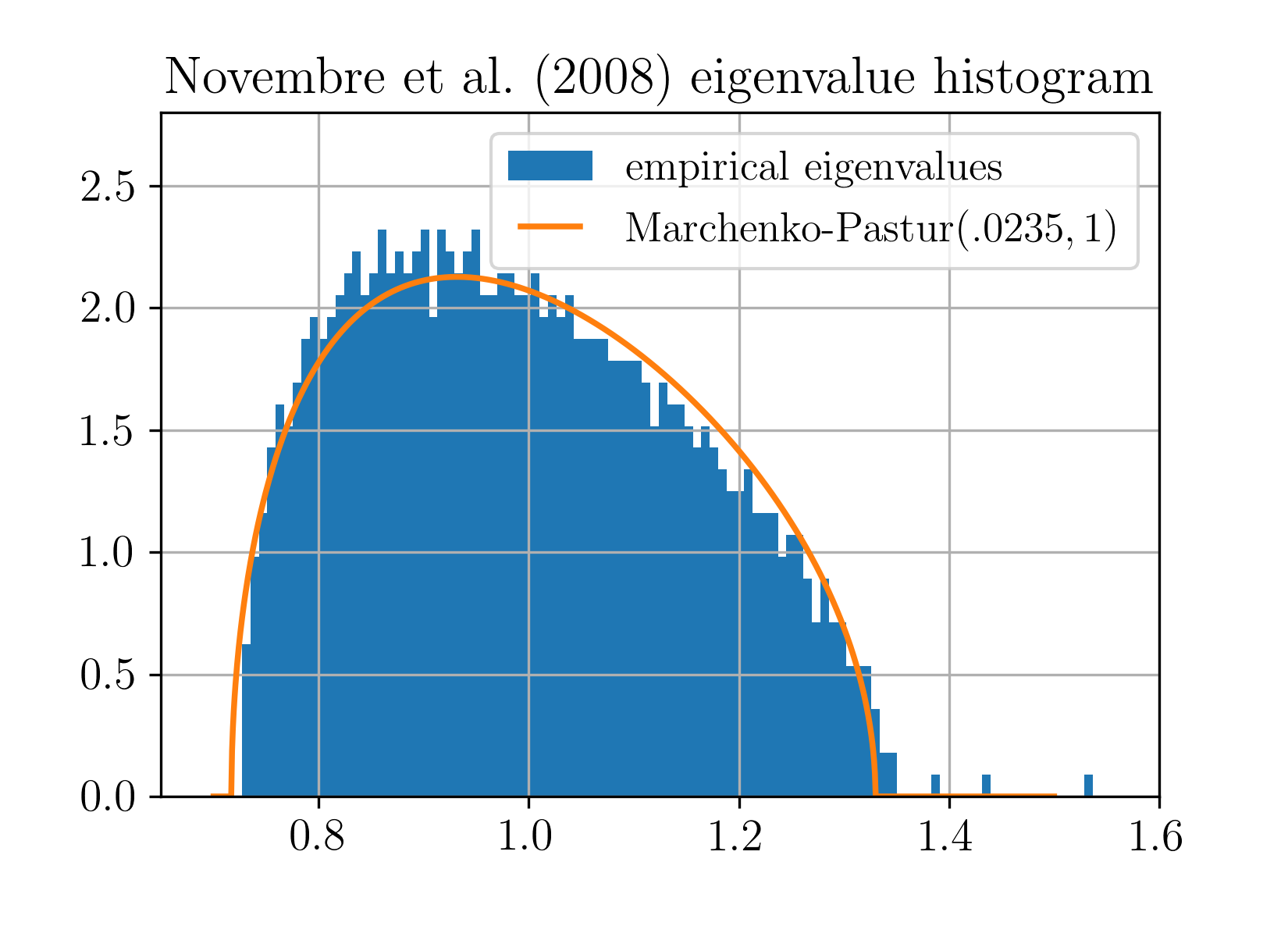}
\caption{The eigenvalues of the sample covariance of data in \cite{N08} plotted in increase order (left) and the histogram of eigenvalues compared to the Marchenko--Pastur distribution  (right).}
\label{fig0}
\end{figure}


\subsection{Notation and Setting}

Given a real matrix $X$ of size $n \times p$, let $(x_{ij}: 1 \leq i \leq n, 1 \leq j \leq p)$ denote the matrix elements, $\sigma_i(X)$ the $i$-th largest singular value, and $u_i(X)$ and $v_i(X)$ the corresponding left and right (unit norm) singular vectors, respectively. If $X$ is symmetric, we write $\lambda_i(X)$ for the $i$-th largest eigenvalue. For a function $f \colon \mathbb{R} \to \mathbb{R} $, let $f(X)$ denote  elementwise application of $f$ to $X$. Let $\odot$ denote the Hadamard product. We use the following norms: (1) the standard  $L^p$ norm on $\mathbb{R}^n$, $\|\cdot\|_p$, (2) the operator and Frobenius matrix norms, $\|\cdot\|_2$ and $\| \cdot \|_F$, and (3) the norm $\|\cdot\|_\mu$ on the Hilbert space $L^2(\mathbb{R},\mu)$ (see Section \ref{sec:ortho}).

We consider two high-dimensional frameworks:

\begin{enumerate}
    \item The {\it asymmetric setting}, 
where $X_n$ and $Z_n$ are real matrices of size $n \times p$ and $p/n \rightarrow \gamma \in (0, \infty)$ as $n \rightarrow \infty$. The signal $X_n$ is deterministic, rank $r$, and has fixed singular values $\sigma_1(X_n) >\cdots >  \sigma_r(X_n)$.  The elements of $Z_n$ are i.i.d. 

\item The {\it symmetric setting}, where  $X_n$ and $Z_n$ are real symmetric matrices of size $n \times n$. The signal $X_n$ is deterministic, rank $r$, and has fixed eigenvalues $\lambda_1(X_n) >\cdots >  \lambda_r(X_n)$. The upper triangular elements of $Z_n$ are i.i.d. We call $Z_n$ a generalized Wigner matrix and $X_n + n^{-1/2}Z_n$ a spiked Wigner matrix. Here, the data has aspect ratio $\gamma = 1$. 
\end{enumerate}

\subsection{Spiked Matrices}

The standard spiked matrix model $Y_n = X_n + n^{-1/2} Z_n$ 
exhibits singular value bias and singular vector inconsistency. The corresponding phase transition depends only the variance of noise: 
\begin{lemma}(Theorem 3.6 of \cite{BS_SpAn}, Theorems 2.8--2.10 of \cite{BGN12}) \label{lem:bgn12} In the asymmetric setting, let $Y_n \coloneqq X_n + n^{-1/2} Z_n$, where the elements of $Z_n$ have mean zero, variance one, and finite moments. The empirical spectral distribution (ESD) of $Y_n^\top Y_n$ converges almost surely weakly to the Marchenko--Pastur law with parameter $\gamma$. Furthermore, 
\begin{align} && \label{qsd1} \hspace{1.2cm}
   \sigma_i^2(Y_n) \xrightarrow[]{a.s.} \lambda(\sigma_i(X_n),\gamma),  && \hspace{.15cm} 1 \leq i \leq r, 
\end{align}
where the biasing function $\lambda(\sigma, \gamma)$ is given by
\[
\lambda(\sigma, \gamma) \coloneqq \begin{dcases}
    \frac{(1+\sigma^2)(\gamma + \sigma^2)}{\sigma^2} & \sigma > \gamma^{1/4} \\
    (1 + \sqrt{\gamma})^2 & \sigma \leq \gamma^{1/4}
\end{dcases} . 
\]
The limiting angles between the singular vectors of $X_n$ and $Y_n$ are 
\begin{equation} \hspace{2.4cm} \label{bgn12_1} 
\begin{aligned}  
    \langle u_i(X_n), u_j(Y_n) \rangle^2 & \xrightarrow{a.s.} \delta_{ij} \cdot c_1^2 (\sigma_i(X_n), \gamma) , \\  \langle v_i(X_n), v_j(Y_n) \rangle^2 & \xrightarrow{a.s.} \delta_{ij} \cdot c_2^2 (\sigma_i(X_n), \gamma) ,  
\end{aligned}   \hspace{1.83cm} 1\leq i,j \leq r , 
\end{equation} 
where the left and right ``cosine" functions $c_1(\sigma,\gamma)$ and $c_2(\sigma,\gamma)$ are given by 
\begin{equation*} 
\begin{aligned} &&
     c_1^2(\sigma,\gamma)  \coloneqq   \begin{dcases}    
     1 - \frac{\gamma+\sigma^2}{\sigma^2(1 + \sigma^2)} & \sigma > \gamma^{1/4} \\ 0 & \sigma \leq \gamma^{1/4} \end{dcases} ,  && \quad\quad
     c_2^2(\sigma,\gamma)  \coloneqq   \begin{dcases}    
     1 - \frac{\gamma(1+\sigma^2)}{\sigma^2(\gamma + \sigma^2)} & \sigma > \gamma^{1/4} \\ 0 & \sigma \leq \gamma^{1/4} \end{dcases} .
\end{aligned} 
\end{equation*}
\end{lemma}
\noindent A phase transition occurs precisely at $\gamma^{1/4}$: above $\gamma^{1/4}$, the ``supercritical" case, singular vectors of $Y_n$ contain information about the corresponding singular vectors of $X_n$.  
Below $\gamma^{1/4}$,  the ``subcritical" case, singular vectors of $Y_n$ are asymptotically orthogonal to the signal. 

\begin{remark} \label{rem_bgn1}
    We note that \cite{BGN12} assumes the singular vectors of $X_n$ are Haar-distributed to ensure convergence to zero of expressions such as $v_i(X_n)^\top(n^{-1} Z_n^\top Z_n - z)^{-1}v_i(X_n) - m_\gamma(z)$, where $m_\gamma(z)$ is the Stieltjes transform of the LSD of $n^{-1} Z_n^\top Z_n$. In our setting, as $Z_n$ has i.i.d.\ elements, this assumption is unnecessary: convergence to zero follows from Theorem 1 of \cite{BMP07} or the isotropic local Marchenko--Pastur law, Theorem 2.4 of \cite{Bloe}.
\end{remark}

\begin{remark} 
    The assumption that non-zero singular values are distinct is for convenience. More generally,
    \[ \sum_{j: \sigma_j(X_n) = \sigma_i(X_n)} \langle u_i(X_n),  u_j(Y_n) \rangle^2 \xrightarrow{a.s.} c_1^2(\sigma_i(X_n),\gamma) , \]
with an analogous statement holding for right singular vectors (see Theorem 2.9 of \cite{BGN12}). 
\end{remark}

\begin{remark} \label{rem_bgn3}
    Theorem 2.10 of \cite{BGN12}, corresponding to the subcritical case of (\ref{bgn12_1}), assumes $r=1$. It is well known, though, that this holds for fixed $r > 1$. Our results formally establish this claim since model (\ref{0}) subsumes (\ref{-1}).  
\end{remark}

The following lemma is the symmetric analog of Lemma \ref{lem:bgn12}. Let $r_+$ denote the number of positive eigenvalues of $X_n$ and $r_- \coloneqq r - r_+$:
\[ \lambda_1(X_n) \geq \cdots \geq \lambda_{r_+}(X_n) > 0  >  \lambda_{n- r_- + 1}(X_n) \geq \cdots \geq \lambda_n(X_n) . \]
\begin{lemma}(Theorem 2.5 of \cite{BS_SpAn}, Theorems 2.1 and 2.2 of \cite{BGN11})
In the symmetric setting,  let $Y_n \coloneqq X_n + n^{-1/2} Z_n$, where the elements of $Z_n$ have mean zero, variance one, and finite moments. The ESD of $Y_n^\top Y_n$ converges almost surely weakly to the semicircle law. Furthermore, 
    \begin{equation}
\begin{aligned}
     \hspace{3.6cm} \lambda_i(Y_n) & \xrightarrow{a.s.} \bar \lambda(\lambda_i(X_n)) , && \hspace{1.9cm} 1 \leq i \leq r_+ , \\
      \hspace{3.6cm} \lambda_{n-i+1}(Y_n) &\xrightarrow{a.s.} \bar \lambda(\lambda_{n-i+1}(X_n)) , &&  \hspace{1.9cm} 1 \leq i \leq r_- ,
\end{aligned}
\end{equation}
where the biasing function $\bar \lambda( \lambda)$ is given by
\begin{align*}
    \bar \lambda(\lambda) \coloneqq \begin{dcases}
        \lambda + \frac{1}{\lambda} & |\lambda| > 1 \\
        2 \textnormal{sign}(\lambda) & |\lambda| \leq 1
    \end{dcases} . 
\end{align*}
The limiting angles between the eigenvectors of $X_n$ and $Y_n$ are
\begin{align}
&&
    \langle v_i(X_n), v_j(Y_n) \rangle^2 \xrightarrow{a.s.} \delta_{ij} \cdot \bar c^2(\lambda_i(X_n)), && i,j \in \{1, \ldots, r_+\} \cup \{n-r_-+1, \ldots, n\},
\end{align}
where the cosine function $\bar c(\lambda)$ is  
\[
\bar c^2 (\lambda) \coloneqq  \begin{dcases}
    1 - \frac{1}{\lambda^2} & |\lambda| > 1 \\ 0 & |\lambda| \leq 1
\end{dcases} . 
\]
\end{lemma}
\begin{remark}
    Analogous statements to Remarks \ref{rem_bgn1}--\ref{rem_bgn3} hold for spiked Wigner matrices.  The relevant result of \cite{Bloe} is Theorem 2.12, the isotropic local semicircle law. 
\end{remark}

\subsection{Orthogonal Polynomials}\label{sec:ortho} In this section, we recall standard properties of orthogonal polynomials.

\begin{lemma} \label{lem:ortho1} ((3.1) of \cite{Geza}) Let $\mu$ be a Borel measure with infinite support and finite moments. There exists a sequence of polynomials $\{q_k\}_{k \in \mathbb{N}_0}$ in the Hilbert space $L^2(\mathbb{R},\mu)$ such that the degree of $q_k$ is $k$ and
\begin{align} 
 \int q_k(z) q_\ell(z) d \mu(z) = \delta_{k \ell} .
\end{align}
This sequence may be constructed by applying the Gram--Schmidt procedure to the monomials $\{z^k\}_{k \in \mathbb{N}_0}$. Requiring that the leading coefficient of each polynomial is positive, $\{q_k\}_{k \in \mathbb{N}_0}$ is unique.   
\end{lemma}

The Hilbert space $L^2(\mathbb{R},\mu)$ is equipped with the following inner product and norm:
\begin{align*} 
&    \langle f, g \rangle_\mu \coloneqq \int f(z) g(z) d\mu(z) , & \|f\|_\mu^2 \coloneqq \langle f, f \rangle_\mu .
\end{align*}

\begin{lemma} \label{lem:ortho2} (Theorem 3.4 and Remark 3.5 of \cite{Ernst}) Let $\mu$ be a Borel measure with infinite support and finite moments. If $\mu$ has finite moment generating function in a neighborhood of zero, the polynomials $\{q_k\}_{k \in \mathbb{N}_0}$ are dense in $L^2(\mathbb{R}, \mu)$: for $f \in L^2(\mathbb{R},\mu)$, define
    \begin{align*} &
    a_k \coloneqq \langle f, q_k \rangle_\mu =  \int f(z) q_k(z)  d\mu(z) ,  && f_K(z) \coloneqq \sum_{k=1}^K a_k q_k(z) .
    \end{align*} Then, $\|f\|_\mu^2 = \sum_{k=0}^\infty a_k^2$ and 
  \begin{align}
       \lim_{K \rightarrow \infty} \int (f(z) - f_K(z))^2   d\mu(z) = 0 . \label{fK} 
       \end{align}
\end{lemma}
\noindent Additional conditions under which (\ref{fK}) holds are given by Theorem 3.5 of \cite{Ernst}. 


\section{Main Results} \label{sec2}
We make the following assumptions: 
\begin{enumerate}[label=(\roman*)] 
   \item The singular vectors of $X_n$ are incoherent with respect to the standard basis:
   \begin{align*}
        \hspace{2.5cm}   \sqrt{n} \cdot \|u_i(X_n)\|_\infty \|v_i(X_n)\|_\infty \longrightarrow 0 ,   \hspace{2.5cm}  1 \leq i \leq r. 
   \end{align*} This implies that the elements of $\sqrt{n} X_n$ uniformly converge to zero. 
   
   \item The law $\mu$ of elements of $Z_n$ has infinite support and finite moments.
   
    \item The transformation  $f \colon \mathbb{R} \to \mathbb{R}$ is polynomially bounded (implying $f \in L^2(\mathbb{R}, \mu)$ by assumption (ii)) and continuous $\mu$-almost everywhere. 
    \item Let $\{q_k\}_{k \in \mathbb{N}_0}$ denote the orthogonal polynomials with respect to $\mu$ guaranteed by Lemma \ref{lem:ortho1}. Defining $a_k \coloneqq \langle f, q_k \rangle_\mu$ and $f_K(z) \coloneqq \sum_{k=1}^K a_k q_k(z)$ as in Lemma \ref{lem:ortho2}, we assume that $a_0 = 0$ and     \end{enumerate}
    \vspace{-.3cm}
     \begin{gather}
        \lim_{K \rightarrow \infty} \int (f(z) - f_K(z))^2   d\mu(z)  = 0 , \label{assumption4.0} \\ 
        \lim_{K \rightarrow \infty} \lim_{x \rightarrow 0}    \frac{1}{x}  \int ( f(x+z) - f_K(x+z)) d\mu(z) = 0 \label{assumption4.1} .      
    \end{gather} 

For a discussion of these conditions, see Remarks \ref{rem1}--\ref{rem4}. In particular, 
if the elements of $Z_n$ are Gaussian, assumption (iii)  directly implies (iv).

Define the coefficients $b_{k} \coloneqq  \langle q_k', 1 \rangle_\mu$ 
and the constant  
    \begin{align} \tau(f,\mu)  \coloneqq  \frac{1}{\|f\|_\mu} \sum_{k=1}^\infty a_k b_k =   \frac{1}{\|f\|_\mu}  \lim_{K\rightarrow \infty}\int f_K'(z) d\mu(z) .\label{tau_def}
    \end{align}
As $\sum_{k=1}^\infty a_kb_k$ converges by Lemma \ref{lem:tau}, $\tau(f,\mu)$ is well defined and the second equality holds (note that $\int f_K'(z) d\mu(z) = \sum_{k=1}^K a_k b_k$). Without loss of generality, we assume that $\tau(f,\mu) \geq 0$ (otherwise, our results apply to $-Y_n$). The spectral properties of $Y_n$ (namely, the limiting behavior of the leading singular values and vectors) 
are asymptotically equivalent to those of a spiked matrix with signal-to-noise ratio $\tau(f,\mu)$:
\begin{theorem}\label{thrm1} Let $Y_n \coloneqq n^{-1/2} f(\sqrt{n} X_n + Z_n)$ and define the matrix
    \[    A_n \coloneqq \tau(f,\mu) \|f\|_\mu X_n + \frac{1}{\sqrt{n}} f(Z_n) . 
    \]
    In the asymmetric or symmetric setting, under assumptions (i)--(iv), 
    \[
    \|Y_n - A_n \|_2 \xrightarrow[]{a.s.} 0 . 
    \]
\end{theorem}

As a consequence of convergence in operator norm, we have the following corollaries:
\begin{corollary} \label{cor1}
In the asymmetric setting, under assumptions (i)--(iv), the ESD of $\|f\|_\mu^{-2} Y_n^\top Y_n$
converges weakly almost surely to the  Marchenko--Pastur law with parameter $\gamma$. 
Furthermore, 
\begin{align} \label{g27v} && \hspace{2.62cm}
\sigma_i^2(Y_n) \xrightarrow[]{a.s.} \lambda(\tau(f,\mu) \sigma_i(X_n), \gamma ) , && \hspace{.72cm}1 \leq i \leq r .   
\end{align}
The limiting angles between the singular vectors of $X_n$ and $Y_n$ are given by
\begin{equation}
 \begin{aligned}  \label{g28v} \hspace{3.75cm}
    \langle u_i(X_n), u_j(Y_n) \rangle^2 & \xrightarrow{a.s.} \delta_{ij} \cdot c_1^2 (\tau(f,\mu) \sigma_i(X_n), \gamma) , \\  \langle v_i(X_n), v_j(Y_n) \rangle^2 & \xrightarrow{a.s.} \delta_{ij} \cdot c_2^2 (\tau(f,\mu)\sigma_i(X_n), \gamma) ,  
\end{aligned}   \hspace{1.65cm} 1\leq i,j \leq r .    
\end{equation}
\end{corollary}

\begin{corollary} \label{cor2}
  In the symmetric setting, under assumptions (i)--(iv), the ESD of $\|f\|_\mu^{-1} Y_n$  
converges weakly almost surely to the  semicircle law. 
Furthermore, 
\begin{equation}
    \begin{aligned}
     \hspace{4.55cm} \lambda_i(Y_n) & \xrightarrow{a.s.} \bar \lambda(\tau(f,\mu)\lambda_i(X_n)) , && \hspace{1.65cm} 1 \leq i \leq r_+ , \\
      \hspace{4.55cm} \lambda_{n-i+1}(Y_n) &\xrightarrow{a.s.} \bar \lambda(\tau(f,\mu) \lambda_{n-i+1}(X_n)) , &&  \hspace{1.65cm} 1 \leq i \leq r_- .
\end{aligned}
\end{equation}
The limiting angles between the eigenvectors of $X_n$ and $Y_n$ are given by
\begin{align}
&&
    \langle v_i(X_n), v_j(Y_n) \rangle^2 \xrightarrow{a.s.} \delta_{ij} \cdot \bar c^2(\tau(f,\mu) \lambda_i(X_n)), && i,j \in \{1, \ldots, r_+\} \cup \{n-r_-+1, \ldots, n\}. 
\end{align}
\end{corollary}


Theorem \ref{thrm1} reveals that PCA is able to recover signal under highly nonlinear or discontinuous transformations. In high dimensions, a phase transition occurs: for signal-to-noise ratios above a threshold---depending simultaneously on  $f$, $\mu$, and $\gamma$---$X_n$ gives rise to outlier singular values in the spectrum of $Y_n$, and the corresponding singular vectors of $Y_n$ contain information about $X_n$. Below this threshold, the singular vectors of $Y_n$ are asymptotically orthogonal to those of $X_n$. 
In contrast, in the standard setting where $X_n + n^{-1/2} Z_n$ is observed directly, the analogous threshold depends only on $\gamma$ and the variance of noise. 

\begin{remark} \label{rem1} The assumption that $a_0 = 0$ is for convenience. Otherwise, in the asymmetric setting, we have \[ \|Y_n - A_n - a_0 {\bf 1}_n {\bf 1}_p^\top \|_2 \xrightarrow{a.s.} 0, \]
where ${\bf 1}_n$ is the all-ones vector of length $n$. The spectrum of $Y_n$ contains an additional outlier eigenvalue, located approximately at $a_0 \sqrt{np}$. This rank-one term may be eliminated by centering the columns of $Y_n$. An analogous statement holds in the symmetric setting.
\end{remark}

\begin{remark} 
    Equation  (\ref{assumption4.0}) holds if $\{q_k\}_{k \in \mathbb{N}_0}$ is dense in $L^2(\mathbb{R},\mu)$ (see Lemma \ref{lem:ortho2}).
\end{remark}

\begin{remark}\label{rem1h23}
  Assumptions (iii) and (\ref{assumption4.0}) imply (\ref{assumption4.1}) if  $\mu$ has differentiable density $\omega$, $\text{supp}(\omega) = \mathbb{R}$, and $\omega'/\omega$ is polynomially bounded (see Lemma \ref{lem:assumption1}). 
\end{remark}

\begin{remark}
  Suppose (1) $f$ is differentiable, (2) $\mu$ has differentiable density $\omega$,  and (3) $\omega'/\omega$ is polynomially bounded. Then, $\tau(f,\mu)$ has the following simple form:
  \begin{align} \tau(f,\mu)= \frac{\E f'(z_{11})}{\|f\|_\mu} . \label{qofg} \end{align} 
  We prove this fact in the appendix, Lemma \ref{lem:tau_f'}. Note that if $f$ is differentiable on a set $B$ with $\mu(B) = 1$, one might expect given (\ref{qofg}) that $\tau(f,\mu) = \|f\|_\mu^{-1} \E[{\bf 1}_B(z_{11}) f'(z_{11})]$. This is not necessarily true: as a counterexample, consider Gaussian noise and $f(z) = {\bf 1}(z \leq 0) - 1/2$. Then, $\int_{\mathbb{R} \backslash \{0\}} f'(z) d\mu(z) = 0$, yet $\tau(f,\mu) \neq 0$ (see Section \ref{sec:ex}).  
\end{remark}

\begin{remark}  \label{rem3}
    Certain measures induce orthogonal polynomials whose derivatives are sums of a finite number of polynomials of the same family, 
    in which case $\tau(f,\mu)$ may simplify. For example, the Hermite polynomials  (an orthogonal basis with respect to the Gaussian measure) satisfy $q_k'(z) = \sqrt{k} q_{k-1}(z)$; hence, $\tau(f,\mu) = \|f\|_\mu^{-1} a_1$.       
\end{remark}

If $\tau(f, \mu) = 0$, Theorem \ref{thrm1} implies that PCA applied to $n^{-1/2} f(\sqrt{n} X_n + Z_n)$ fails. In this situation, the critical scaling of $X_n$ under which a phase transition occurs is  $n^{1-1/(2\ell_*)}$, where $\ell_* > 1$ is defined as follows: let 
$b_{k \ell} \coloneqq  \langle q_k^{(\ell)}, 1 \rangle_\mu$,
    \begin{align*} \tau_\ell(f,\mu)  \coloneqq  \frac{1}{\|f\|_\mu} \sum_{k=\ell}^\infty a_k b_{k\ell} =   \frac{1}{\|f\|_\mu}  \lim_{K\rightarrow \infty}\int f_K^{(\ell)}(z) d\mu(z) ,  \end{align*}
    and 
    \[ \phantom{.}
\ell_* \coloneqq \inf \{\ell \in \mathbb{N}: \tau_\ell(f,\mu) \neq 0\} .
\]  As 
$\sum_{k=1}^\infty a_k b_{k\ell}$ converges by Lemma \ref{lem:tau2}, $\tau_\ell(f,\mu)$ is well defined.  Note that $\tau_1(f,\mu) = \tau(f,\mu)$.  Without loss of generality, we assume that $\tau_{\ell^*}(f,\mu) > 0$ (otherwise, our results apply to $-Y_n$).  

 To make precise statements, we require stronger assumptions than (i)--(iv):

\begin{enumerate}[label=(\roman*$'$)] 
   \item $X_n$ is rank one and the elements of $n^{1-1/(2\ell_*)} X_n$ uniformly converge to zero. For even $\ell \in \mathbb{N}$, the empirical moments of the elements of $\sqrt{n} u_1(X_n)$ and $\sqrt{p} v_1(X_n)$ converge: 
\begin{align*}
    & m_\ell^u  \coloneqq \lim_{n \rightarrow \infty} \frac{1}{n} \| \sqrt{n} u_1(X_n)\|_\ell^\ell , & &
     m_\ell^v \coloneqq \lim_{n \rightarrow \infty} \frac{1}{p} \| \sqrt{p} v_1(X_n)\|_\ell^\ell  .  
\end{align*}
   \item The law $\mu$ has probability density $\omega$, $\text{supp}(\omega) = \mathbb{R}$,  and $\omega$ has finite  moment generating function in a neighborhood of zero. Additionally, $\omega^{(\ell)}$ exists and $\omega^{(\ell)}/\omega$ is polynomially bounded for $\ell \leq \ell_*$. 
    \item The transformation  $f \colon \mathbb{R} \to \mathbb{R}$ is polynomially bounded and almost everywhere continuous  with respect to the Lesbesgue measure. Additionally,  $a_0 = \langle f, q_0 \rangle_\mu = 0$.
    \item Let $\ell_* < \infty$. For $\ell < \ell_*$, we assume that for all sufficiently large integers $K$,
    \[
    \sum_{k=\ell}^K a_k b_{k\ell} = 0 . 
    \]
    Recall that by definition of $\ell_*$, $\sum_{k=\ell}^\infty a_k b_{k\ell} = 0$ for $\ell < \ell_*$.
    \end{enumerate}
See Remarks \ref{rem26}--\ref{:(} for comments on these assumptions (particularly (iv$\mathrm{'}$), which is rather opaque).

We now state our extension of Theorem \ref{thrm1}:
\begin{theorem}
\label{thrm2} Let $Y_n \coloneqq n^{-1/2} f(n^{1-1/(2\ell_*)} X_n + Z_n)$ and define the normalized vectors 
\begin{align*}
  &   \tilde u_1 \coloneqq \frac{(u_1(X_n))^{\odot \ell_*}}{\|(u_1(X_n))^{\odot \ell_*}\|_2} , && \tilde v_1 \coloneqq \frac{(v_1(X_n))^{\odot \ell_*}}{\|(v_1(X_n))^{\odot \ell_*}\|_2} ,
\end{align*} and the matrix
    \[    A_{\ell_*, n} \coloneqq \frac{\tau_{\ell_*}(f,\mu) \sqrt{m_{2\ell_*}^u m_{2\ell_*}^v}  \|f\|_\mu}{\ell_*! \gamma^{(\ell_* -1)/2}} \sigma_1^{\ell_*}(X_n) \tilde u_1 \tilde v_1^\top + \frac{1}{\sqrt{n}} f(Z_n) . 
    \]
    In the asymmetric or symmetric setting, under assumptions (i$\mathrm{'}$)--(iv$\mathrm{'}$), 
    \[
    \|Y_n - A_{\ell*, n} \|_2 \xrightarrow[]{a.s.} 0 . 
    \]
\end{theorem}

\noindent  
Introducing the shorthand notation $\tilde \tau_{\ell_*} \coloneqq \tau_{\ell_*}(f,\mu) \sqrt{m_{2\ell_*}^u m_{2\ell_*}^v} / (\ell_*! \gamma^{(\ell_*-1)/2})$, we have by Theorem \ref{thrm2} the following analogs of Corollaries \ref{cor1} and \ref{cor2}: 

\begin{corollary} \label{cor1.1}
In the asymmetric setting, under assumptions (i$\mathrm{'}$)--(iv$\mathrm{'}$), the ESD of $\|f\|_\mu^{-2} Y_n^\top Y_n$  
converges weakly almost surely to the  Marchenko--Pastur law with parameter $\gamma$. 
Furthermore, 
\begin{align}  
\sigma_1^2(Y_n) \xrightarrow[]{a.s.} \lambda(\tilde \tau_{\ell_*} \sigma_1^{\ell_*}(X_n), \gamma) .
\end{align}
The limiting angles between the first singular vectors of $Y_n$ and $\tilde u_1$ and $\tilde v_1$ are given by
\begin{equation}
 \begin{aligned}  
  \langle \tilde u_1, u_1(Y_n) \rangle^2 & \xrightarrow{a.s.}  c_1^2 (\tilde \tau_{\ell_*} \sigma_1^{\ell_*}(X_n), \gamma) , \\   \langle \tilde v_1, v_1(Y_n) \rangle^2 & \xrightarrow{a.s.}  c_2^2 (\tilde \tau_{\ell_*} \sigma_1^{\ell_*}(X_n), \gamma) .   
\end{aligned}      
\end{equation}
\end{corollary}

\begin{corollary} \label{cor2.1}
  In the symmetric setting, under assumptions (i$\mathrm{'}$)--(iv$\mathrm{'}$), the ESD of $\|f\|_\mu^{-1} Y_n$  
converges weakly almost surely to the  semicircle law. 
Furthermore, assuming $\tilde \tau_{\ell_*} > 0$ and $r_+ = 1$ for simplicity,
\begin{equation}
    \begin{aligned}
      \lambda_1(Y_n) & \xrightarrow{a.s.} \bar \lambda(\tilde \tau_{\ell_*} \lambda_1^{\ell_*}(X_n)) .
\end{aligned}
\end{equation}
The limiting angle between the first eigenvector of $Y_n$ and $\tilde v_1$ is given by
\begin{align}
    \langle \tilde v_1, v_1(Y_n) \rangle^2 \xrightarrow{a.s.}  \bar c^2(\tilde \tau_{\ell_*} \lambda_1^{\ell_*}(X_n)). 
\end{align}
\end{corollary}

For certain nonlinear transformations and noise distributions, Theorem \ref{thrm2} shows that PCA applied to $n^{-1/2} f(\sqrt{n} X_n + Z_n)$ fails. The critical scaling of $X_n$ under which a phase transition occurs is  $n^{1-1/(2\ell_*)}$, and the transition point depends on $f$, $\omega$,  and $\gamma$ (as in Theorem \ref{thrm1}) and the empirical moments of $u_1(X_n)$ and $v_1(X_n)$. Above this transition, the singular vectors of $Y_n = n^{-1/2} f(n^{1-1/(2\ell_*)} X_n + Z_n)$ contain information about $X_n$, while below the transition, the output of PCA is asymptotically orthogonal  to the signal.

\begin{remark} \label{rem26}
    Although we assume $X_n$ is deterministic, all results naturally generalize (by conditioning on $X_n$) to the case of random $X_n$ independent of $Z_n$. 
    Assumption (i$\mathrm{'}$) is satisfied if the elements of $\sqrt{n} u_i(X_n)$ 
    (and those of $\sqrt{p} v_i(X_n)$) are i.i.d.\ variables with variance one and finite moments, in which case $n^{-1} \|\sqrt{n} u_i(X_n)\|_k^k$ converges to the $k$-th moment, almost surely. Although $u_i(X_n)$ is only asymptotically unit norm, as $X_n$ is fixed rank, this effect is negligible.  
\end{remark}

\begin{remark} \label{rem4} Assumption (i$\mathrm{'}$) implies (i), assumption (ii$\mathrm{''}$) and Lemma \ref{lem:ortho2} imply (\ref{assumption4.0}), and assumptions (ii$\mathrm{'}$) and (iii$\mathrm{'}$) imply (\ref{assumption4.1}) (see Lemma \ref{lem:assumption1}). 
\end{remark}

\begin{remark} \label{:(} Assumption (iv$\mathrm{'}$) is admittedly restrictive. Nevertheless, two important cases are covered in which the sum $\sum_{k=\ell}^\infty a_k b_{k \ell}$ contains a finite number of non-zero terms for $\ell < \ell_*$.  First, polynomial transformations are included: if $f$ is a degree $m$ polynomial, $a_k = 0$ for $k > m$.  Second, assumption (iv$\mathrm{'})$ holds if there exists $m \in \mathbb{N}$ and coefficients $(\alpha_{kj} : k \in \mathbb{N}, k-m \leq j \leq k-1)$ such that
\begin{align}
q_k'(z) = \sum_{j=k-m}^{k-1} \alpha_{kj} q_j(z) ,  \label{:'(}
\end{align}
for $k \geq m$. This implies $b_{k\ell} = 0$ for $k$ sufficiently large and each $\ell < \ell_*$. The Hermite polynomials (corresponding to the Gaussian distribution) satisfy such a recurrence with $m=1$  (Remark \ref{rem3}). A characterization of distributions such that the corresponding orthogonal polynomials satisfy (\ref{:'(}) with $m=2$ is given in \cite{Bonan}.

At the point in the proof of Theorem \ref{thrm2} where assumption (iv$\mathrm{'}$) is used, we briefly discuss a more detailed method of analysis that could eliminate this technical condition. 

\end{remark}

\begin{remark} \label{rem210}
    For $\ell_* = 1$, Theorem \ref{thrm2} reduces (as expected) to a rank-one specialization of Theorem \ref{thrm1}; $u_1(X_n)$ and $v_1(X_n)$ are unit norm, so $m_{2\ell_*}^u = m_{2\ell_*}^v = 1$ and $\tilde \tau_{\ell_*} = \tau(f,\mu)$.
\end{remark}




\section{Applications} \label{sec:ex}


\noindent \textbf{Binomial data.} We consider binomial data with latent low-rank structure:
\begin{align} \label{d3x}  y_{ij} \sim \text{Bin}(m, \text{logistic}(x_{ij})),   
\end{align}
where $\text{logistic}(x) \coloneqq (1+e^{-x})^{-1}$. 
As noted in Section \ref{sec:moti}, (\ref{d3x}) is representable  as a transformed spiked matrix: if $z_{ij}$ is logistically distributed, 
\begin{align} \label{plmokn2}
{\bf 1}(-x_{ij} + z_{ij} \leq 0) \sim \text{Ber}(\text{logistic}(x_{ij})) 
\end{align}
(summing $m$ i.i.d.\ copies of (\ref{plmokn2}) and mean-centering yields (\ref{d3x})).

\begin{corollary} \label{thrm_bin}
    Let $Y_n$ have elements distributed according to 
    \begin{align} \sqrt{n} y_{ij} \sim \mathrm{Bin}(m, \mathrm{logistic}(\sqrt{n/m} \, x_{ij})) - \frac{m}{2}  \label{f4q} \end{align}
    and the elements of $ \sqrt{n/m} \, X_n$ uniformly converge to zero. In the asymmetric setting, the ESD of $4m^{-1} Y_n^\top Y_n$ converges weakly almost surely to the Marchenko--Pastur law with parameter $\gamma$. The limiting angles between the singular vectors of $X_n$ and $Y_n$ are 
\begin{equation}
 \begin{aligned}  \label{g-1v} \hspace{3.7cm}
    \langle u_i(X_n), u_j(Y_n) \rangle^2 & \xrightarrow{a.s.} \delta_{ij} \cdot c_1^2 (\sigma_i(X_n)/2, \gamma) , \\  \langle v_i(X_n), v_j(Y_n) \rangle^2 & \xrightarrow{a.s.} \delta_{ij} \cdot c_2^2 (\sigma_i(X_n)/2, \gamma) ,  
\end{aligned}   \hspace{1.5cm} 1\leq i,j \leq r .    
\end{equation}
   
\end{corollary}

\begin{remark}
    \label{rem:bin} The scaling factors in Corollary \ref{thrm_bin} are for consistency with Theorem \ref{thrm1} and Corollary \ref{cor1}. Under model (\ref{d3x}), assuming the elements of $X_n$ converge uniformly to zero, the recovery threshold of PCA is $2  \sqrt{n/m} \gamma^{1/4}$:
\begin{align*}
 \liminf_{n \rightarrow \infty} \, \langle u_i(X_n), u_{i+1}(Y_n) \rangle^2 > 0 \quad \text{ if and only if } \quad  \liminf_{n \rightarrow \infty} \sqrt{\frac{m}{n}} \sigma_i(X_n) > 2 \gamma^{1/4} ,  \quad \quad 1 \leq i \leq r,
\end{align*}
with an identical threshold for right singular vectors. Here, we estimate $u_i(X_n)$ by $u_{i+1}(Y_n)$ as $Y_n$ is non-centered and its leading eigenvalue is non-informative (see Remark \ref{rem1}). The number of trials $m$ may be fixed ($m = 1, 2$ are particular cases of interest) or increase with $n$.  In Figure \ref{fig1}, we simulate binomial data from (\ref{d3x}) with $m = 2$ and $m= \lfloor \sqrt{n} \rfloor$; empirical cosine similarities agree closely with theory. 
\end{remark}

We note that under (\ref{d3x}), the degree of heteroskedasticity in $Y_n$ is slight; success probabilities uniformly converge to one-half (or to $\alpha \in (0,1)$, by taking the noise to have non-zero mean). Although this assumption is simple, Corollary \ref{thrm_bin} shows it induces a novel phase transition. Increasing the level of heteroskedasticity, (1) our proof method breaks down and precise asymptotic results may no longer be attainable, and (2) estimation of $\text{rank}(X_n)$ now poses a challenge: the ESD of $m^{-1}Y_n^\top Y_n$ converges to an unknown, signal-dependent distribution 
 rather than the Marchenko--Pastur law.\footnote{Under convergence of the ESD to the Marchenko--Pastur law, the number of supercritical singular values of $X_n$ is consistently estimated by the number of eigenvalues of $4m^{-1} Y_n^\top Y_n$ exceeding $(1+\sqrt{\gamma})^2+n^{-2/3+\varepsilon}$, where $\varepsilon \in (0, 2/3)$.} This setting is studied in \cite{landa}, which develops a new whitening procedure for estimation of $\text{rank}(X_n)$ within a similar model.

\begin{figure}[]
\centering
\includegraphics[height=2.425in]{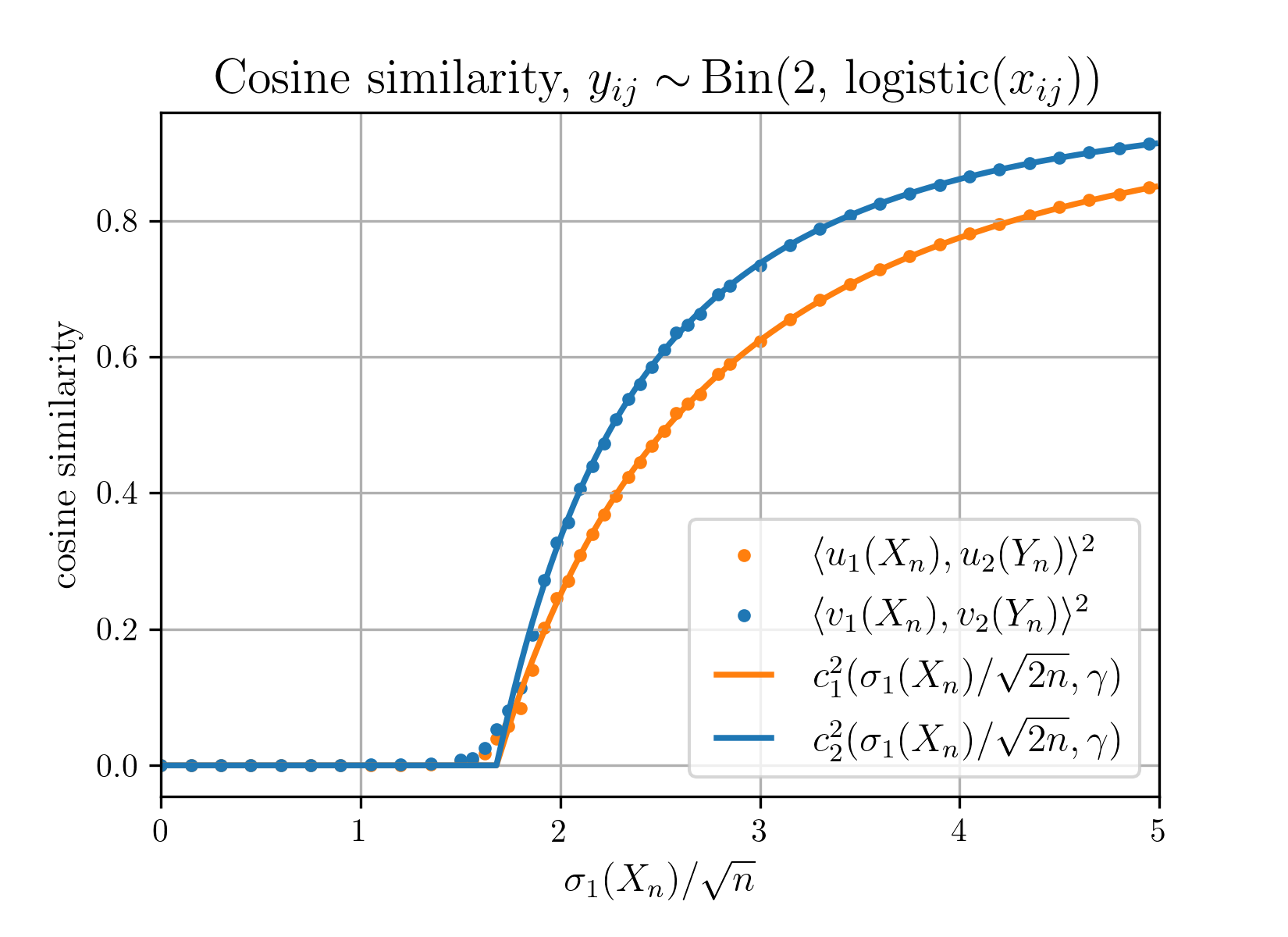}\hspace{-.6cm}
\includegraphics[height=2.425in]{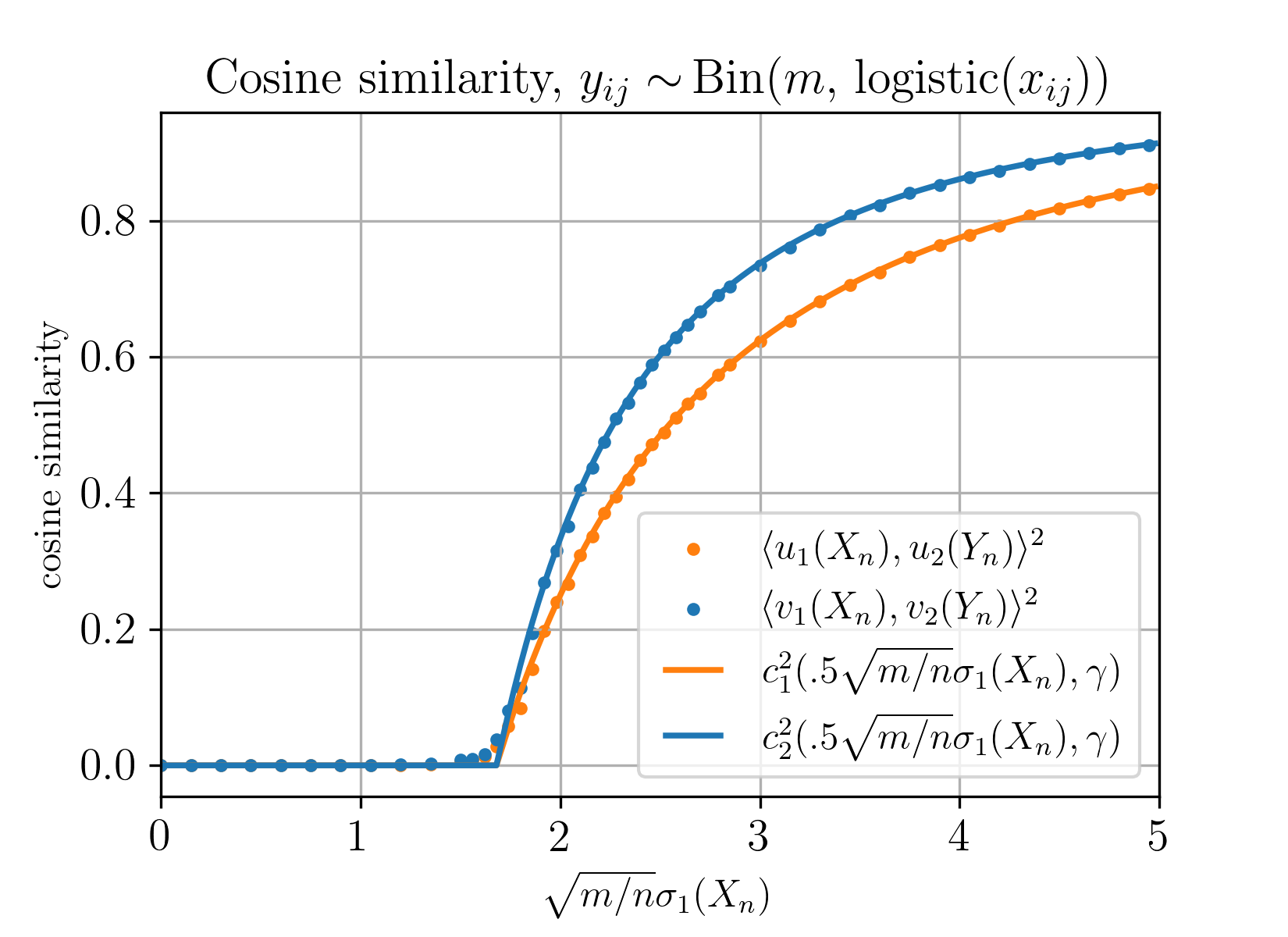}
\caption{ Cosine similarities between the singular vectors of $X_n$ and $Y_n$ under (\ref{d3x}), with $n = 5000$, $p = 2500$, $\gamma = 1/2$, and $m = 2$ (left) or $m = \lfloor \sqrt{n} \rfloor$ (right). The singular vectors of $X_n$ were generated uniformly on the unit sphere. There is close agreement between theory (solid lines) and simulations (points, each representing the average  25 simulations).}
\label{fig1}
\end{figure}

\vspace{.1in}
\noindent {\bf Optimal elementwise preprocessing.} If $X_n + n^{-1/2} Z_n$ is observed directly, suitable elementwise preprocessing improves the performance of PCA: 


\begin{corollary} \label{cor3}
Let assumptions (i)--(ii) hold, $\E z_{11} = 0$, $\textnormal{Var}(z_{11}) = 1$, and $\sum_{k=1}^\infty b_{k}^2 > 1$. 
Application of $f_K^*(z) \coloneqq \sum_{k=1}^K b_1 q_k(z)$ to $\sqrt{n}X_n + Z_n$ reduces the recovery threshold of PCA from $\gamma^{1/4}$ to \[ \gamma^{1/4} \tau(f_K^*, \mu)^{-1} = \gamma^{1/4}\bigg(\sum_{k=1}^K b_k^2 \bigg)^{-1/2}.\] 

Moreover, if $\mu$ has differentiable density $\omega$,  $\textnormal{supp}(\omega) = \mathbb{R}$,  and $\omega'/\omega$ is polynomially bounded, $f_K^*(z) \rightarrow f^*(z) \coloneqq -\omega'(z)/\omega(z)$, where the convergence is in $L^2(\mathbb{R},\mu)$. In this case, $f^*$ maximizes $\tau(f,\mu)$ and  
\begin{align} \label{qwerty6} 
\tau^2(f^*, \mu) = 
 \sum_{k=1}^\infty b_k^2 = \mathcal{I}(\omega) \coloneqq \int_{-\infty}^\infty  \frac{\omega'(z)^2}{\omega(z)} dz   \geq 1, 
\end{align}
where $\mathcal{I}(\omega)$ is the Fisher information under translation. The inequality is strict if and only if $z_{11}$ is not Gaussian.    
Transformation by $f^*$ reduces the recovery threshold of PCA from $\gamma^{1/4}$ to $ \gamma^{1/4} \tau(f^*, \mu)^{-1} = \gamma^{1/4} \mathcal{I}(\omega)^{-1/2}$.
\end{corollary}

\noindent This corollary extends results of \cite{Perry} and \cite{Mont}, which assume $\mu$ has differentiable density $\omega$ and study the optimal transformation $f^*$.\footnote{\cite{Mont} considers kernel density estimation of $\omega$, which is not assumed known as in this work.} Used together with Corollaries \ref{cor1} and \ref{cor2}, Corollary \ref{cor3} yields the limiting cosine similarities between the singular vectors of $X_n$ and $f_K^*(\sqrt{n}X_n+Z_n)$ or  $f^*(\sqrt{n}X_n+Z_n)$ (\cite{Perry} provides only a lower bound on the limiting cosine similarity in the symmetric setting). Our results also offer a new perspective on $f^*$, as the (almost-everywhere) unique  maximizer of $\tau(f,\mu)$. 

\begin{example} \label{ex3} Let $\phi(z) \coloneqq (2 \pi)^{-1/2} \exp(-z^2/2)$ denote the standard Gaussian density. Suppose we observe data $X_n + n^{-1/2} Z_n$,  where the elements of $Z_n$ have a bimodal distribution with density 
\[
\omega(z) = \frac{1}{2} \big( \phi(2(z-1)) + \phi(2(z+1)) \big). 
\]    
Applying the optimal elementwise-transformation $f(z) = -\omega'(z)/\omega(z)$ of Corollary \ref{cor3} reduces the recovery threshold of PCA from $1$ to $\mathcal{I}(\omega)^{-1/2}$, with 
\begin{align}
    \mathcal{I}(\omega) =\int_{-\infty}^\infty \frac{\omega'(z)^2}{\omega(z)} dz  \approx 2.902.
\end{align}
By Corollary \ref{cor2},
\begin{equation}
\begin{aligned} &
    \langle u_1(X_n), u_1(Y_n)  \rangle^2 \xrightarrow{a.s.} \bar c^2\big(\sqrt{\mathcal{I}(\omega)} \sigma_1(X_n)\big)  .
\end{aligned}  
\end{equation}
In Figure \ref{fig2}, we see that transforming the data by $f^*$ significantly improves the performance of PCA.  
\end{example}

\begin{figure}[]
\centering
\includegraphics[height=2.4in]{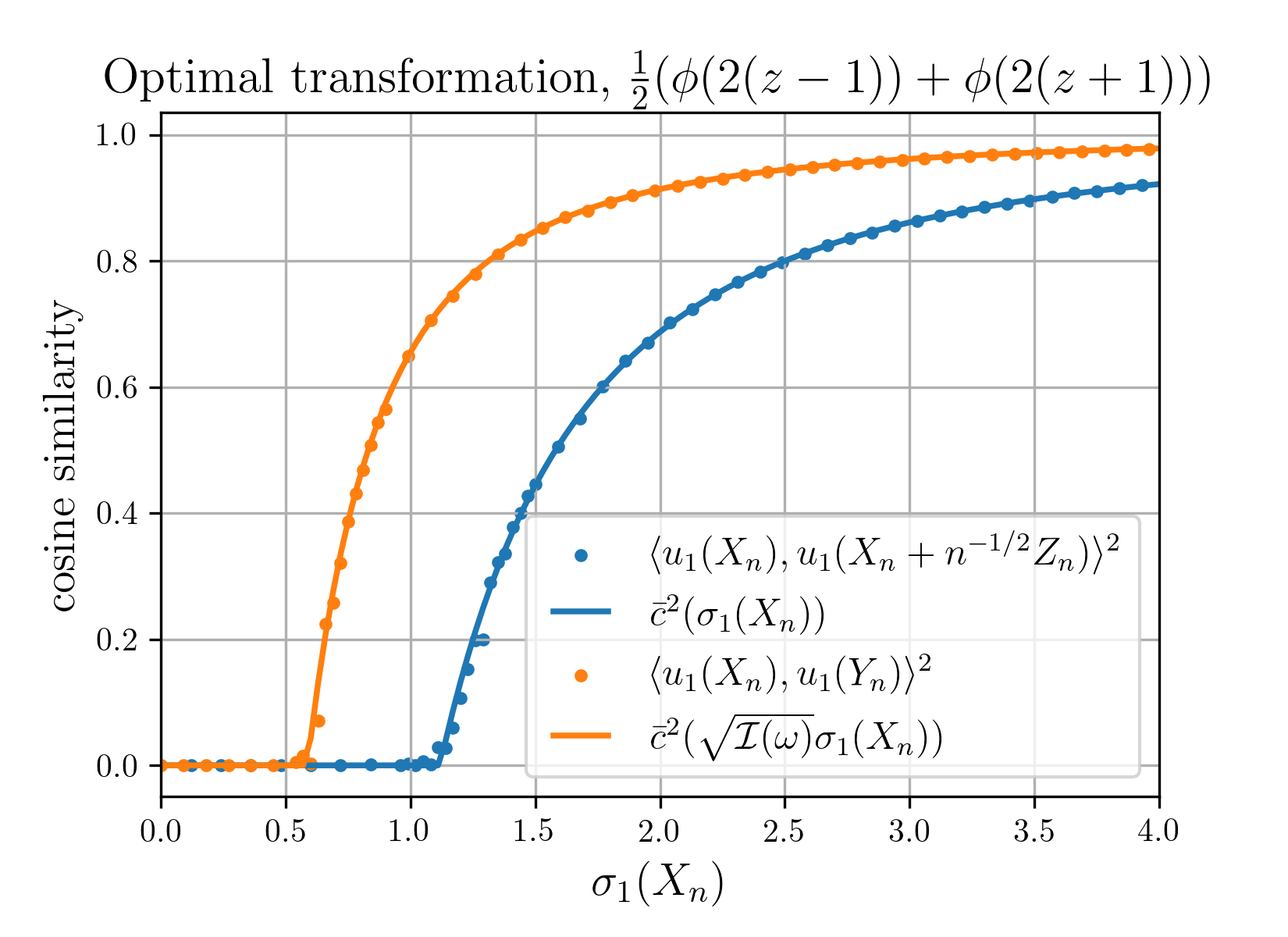} \hspace{-.6cm}
\includegraphics[height=2.4in]{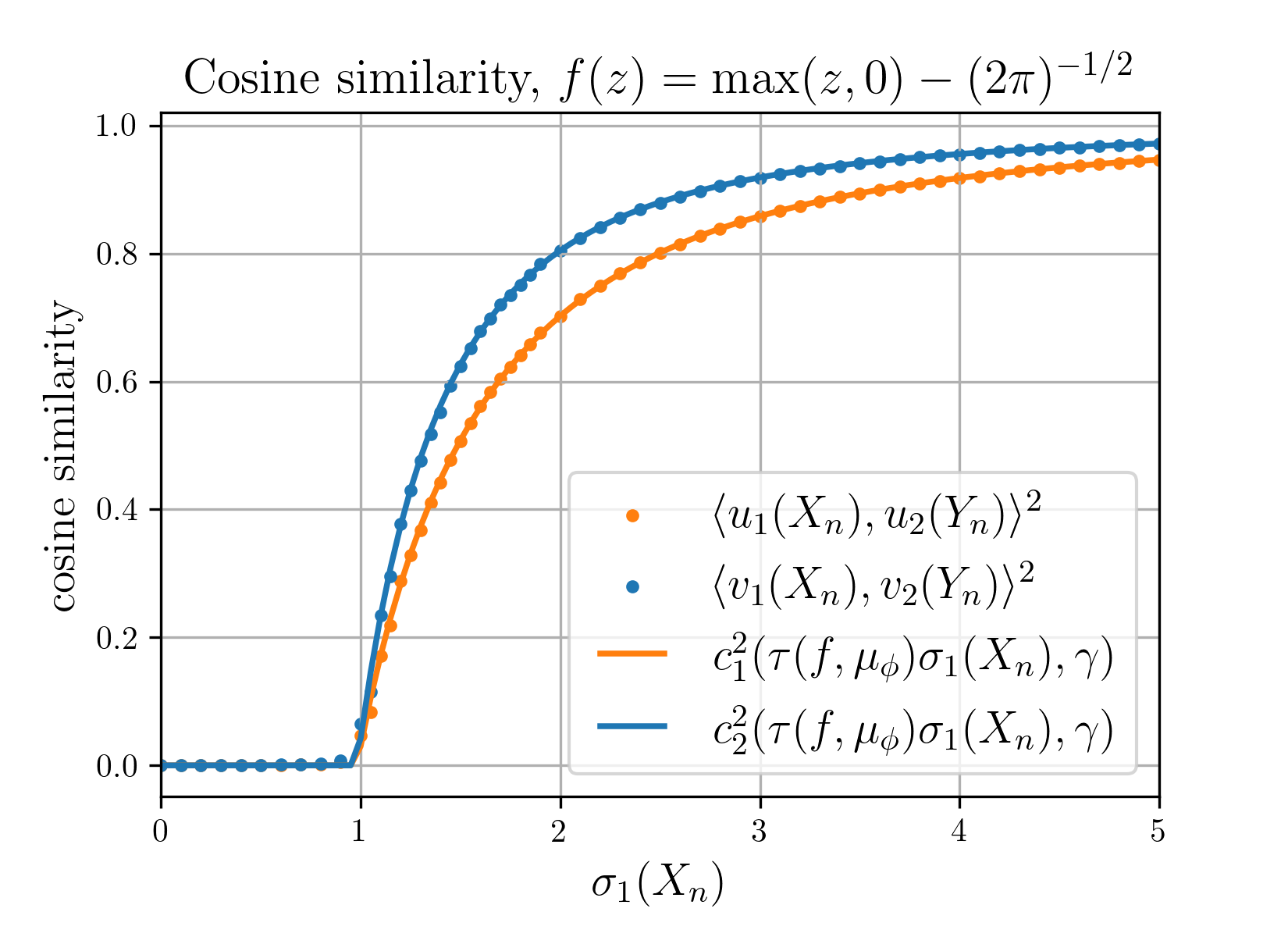}
\caption{ {\bf Left}: Cosine similarities in the symmetric setting between the eigenvectors of $X_n$ and $X_n + n^{-1/2} Z_n$ (blue) and $X_n$ and $Y_n$ (orange). The elements of $Z_n$ have a bimodal distribution, $n = 5000$, and $Y_n = n^{-1/2} f^*(X_n + n^{-1/2} Z_n)$ where $f^*$ is the transformation introduced in Corollary \ref{cor3}. Application of $f^*$ reduces the recovery threshold of PCA from $1$ to roughly $.587$. \\
{\bf Right}: Cosine similarities between the singular vectors of $X_n$ and $Y_n$ with $f(z) = \max(z,0) - (2\pi)^{-1/2}$, the ReLU function, $n = 5000$, $p = 2500$, and $\gamma = 1/2$. This transformation increases the recovery threshold of PCA from $\gamma^{1/4} \approx .841$ to $\gamma^{1/4} \tau(f,\mu_\phi) \approx .982$. 
\\In both plots, there is close agreement between theory (solid lines) and simulations (points, each representing the average  25 simulations).}
\label{fig2}
\end{figure}

\vspace{.1in}

\noindent {\bf ReLU activation.} Consider  the ReLU function $f(z) = \max(z,0) - (2\pi)^{-1/2}$ and Gaussian noise  ($f(z)$ is centered so that $\E f(z_{11}) = 0$, see Remark \ref{rem1}). This is a particular form of missing data: negative data elements are unobserved. 

Let  $\mu_\phi$ denote the Gaussian measure. Using Remark \ref{rem3}, we explicitly calculate $\tau(f,\mu_\phi)$:
\begin{align} \label{relu_tau}
\tau(f,\mu_\phi) = \|f\|_{\mu_\phi}^{-1} a_1 = \bigg( \int_{-\infty}^\infty f^2(z) \phi(z) dz \bigg)^{-
1/2} \int_{-\infty}^\infty z f(z) \phi(z) dz = \sqrt{\frac{\pi}{2(\pi-1)}} . 
\end{align}
 In the asymmetric setting, by Corollary \ref{cor1}, we have
\begin{equation}
\begin{aligned} &
  \langle u_1(X_n), u_1(Y_n)  \rangle^2 & \xrightarrow{a.s.} c_1^2(\tau(f,\mu_\phi) \sigma_1(X_n), \gamma), &&\quad   \langle v_1(X_n), v_1(Y_n)  \rangle^2  &\xrightarrow{a.s.} c_2^2(\tau(f,\mu_\phi) \sigma_1(X_n), \gamma) . \label{ex1.1}
\end{aligned}  
\end{equation}
This scenario is simulated in the left-hand panel of Figure \ref{fig3}. 

\vspace{.1in}

\noindent {\bf Truncated data.} Observed data may be inherently truncated, or truncation may be intentionally applied to the data as a preprocessing step. The effect of truncation on PCA depends heavily on the distribution of noise---under Gaussian noise, truncation raises the recovery threshold, while under heavy-tailed noise, truncation may dramatically lower the recovery threshold. For a given noise distribution, we are able to calculate the optimal thresholding level. 

Corollary \ref{cor:trunc} below does not directly follow from results in Section \ref{sec2} as heavy-tailed distributions may lack finite moments, violating assumption (ii). Rather, we demonstrate that by specializing the proof of Theorem \ref{thrm1} to thresholding transformations, assumptions (ii)--(iv) may be relaxed. See Example \ref{ex69}, in which we study the spiked model with Cauchy-distributed noise. 

\begin{corollary} \label{cor:trunc}
    Let $f_c(z) \coloneqq z {\bf 1} (|z| \leq c)$ denote truncation at level $c$. We assume  the elements of $\sqrt{n}X_n$ uniformly converge to zero and $\mu$ has density $\omega$ that is continuous at $\pm c$. Let $Y_n \coloneqq n^{-1/2} f_c(\sqrt{n} X_n + Z_n)$, let $F_\mu$ denote the cumulative distribution function  of $\mu$, and  define
    \begin{align} \label{truncato}
    \tau(f_c, \mu) \coloneqq \frac{F_\mu(c) - F_\mu(-c) - c(\omega(c) + \omega(-c))} {\sqrt{\textnormal{Var}(f_c(z_{11}))}} ,
    \end{align}
     and    $A_n \coloneqq \tau(f_c, \mu) {\sqrt{\textnormal{Var}(f_c(z_{11}))}} X_n + n^{-1/2}f_c(Z_n)$. Then, in the asymmetric or symmetric setting,   
    \[
        \|Y_n - A_n\|_2 \xrightarrow{a.s.} 0 .
    \]
\end{corollary}

\begin{remark}
    For measures $\mu$ additionally satisfying assumptions (ii)--(iv), definition (\ref{truncato}) is consistent with $\tau(f_c,\mu)$ as defined in (\ref{tau_def}). \end{remark}

 Used together with Corollaries \ref{cor1} and \ref{cor2}, Corollary \ref{cor:trunc} yields the limiting cosine similarities between the singular vectors of $X_n$ and $Y_n$. In particular, the recovery threshold of PCA is $\tau(f_c,\mu)^{-1}\gamma^{1/4}$: 
\begin{align*}
\liminf_{n \rightarrow \infty} \, \langle u_i(X_n), u_{i}(Y_n - \E f_c(z_{11}) {\bf 1}_n {\bf 1}_p^\top) \rangle^2 > 0 \quad \text{ if and only if } \quad  \liminf_{n \rightarrow \infty} \tau(f_c,\mu) \sigma_i(X_n) > \gamma^{1/4} ,  \quad  1 \leq i \leq r.
\end{align*}

\begin{example}  Under Gaussian noise,  using (\ref{truncato}),
   \[
  \tau(f_c,\mu_\phi) = \sqrt{\text{erf}\Big(\frac{c}{\sqrt{2}}\Big) - 2 c \phi(c) }.
 \]  
 Note that $\tau(f_c, \mu_\phi) < 1$ in accordance with Corollary \ref{cor3}. 
\end{example}

\begin{example} \label{ex69} Under Cauchy-distributed noise (let $\mu$ have density $\omega(z) = (1+z^2)^{-1}$), the LSD of $n^{-1} Z_n^\top Z_n$ is heavy-tailed. As a result, the leading singular vectors of $\sqrt{n}X_n + Z_n$ are orthogonal to those of $X_n$ and PCA is ineffective. 
Truncating, $f_c(\sqrt{n} X_n + Z_n)$ is approximately a spiked matrix with signal-to-noise ratio 
    \begin{align*}
        \tau(f_c,\mu) = \frac{\sqrt{2} \big(\arctan(c) - c(1+c^2)^{-1}\big)}{\sqrt{\pi(c - \arctan(c))}} .
    \end{align*}
     The optimal thresholding level is $c^* = \text{argmax}_c \tau(f_c,\mu) \approx 2.028$.  This example is simulated in Figure \ref{fig3}.
\end{example}

\begin{figure}[]
\centering 
\includegraphics[height=2.4in]{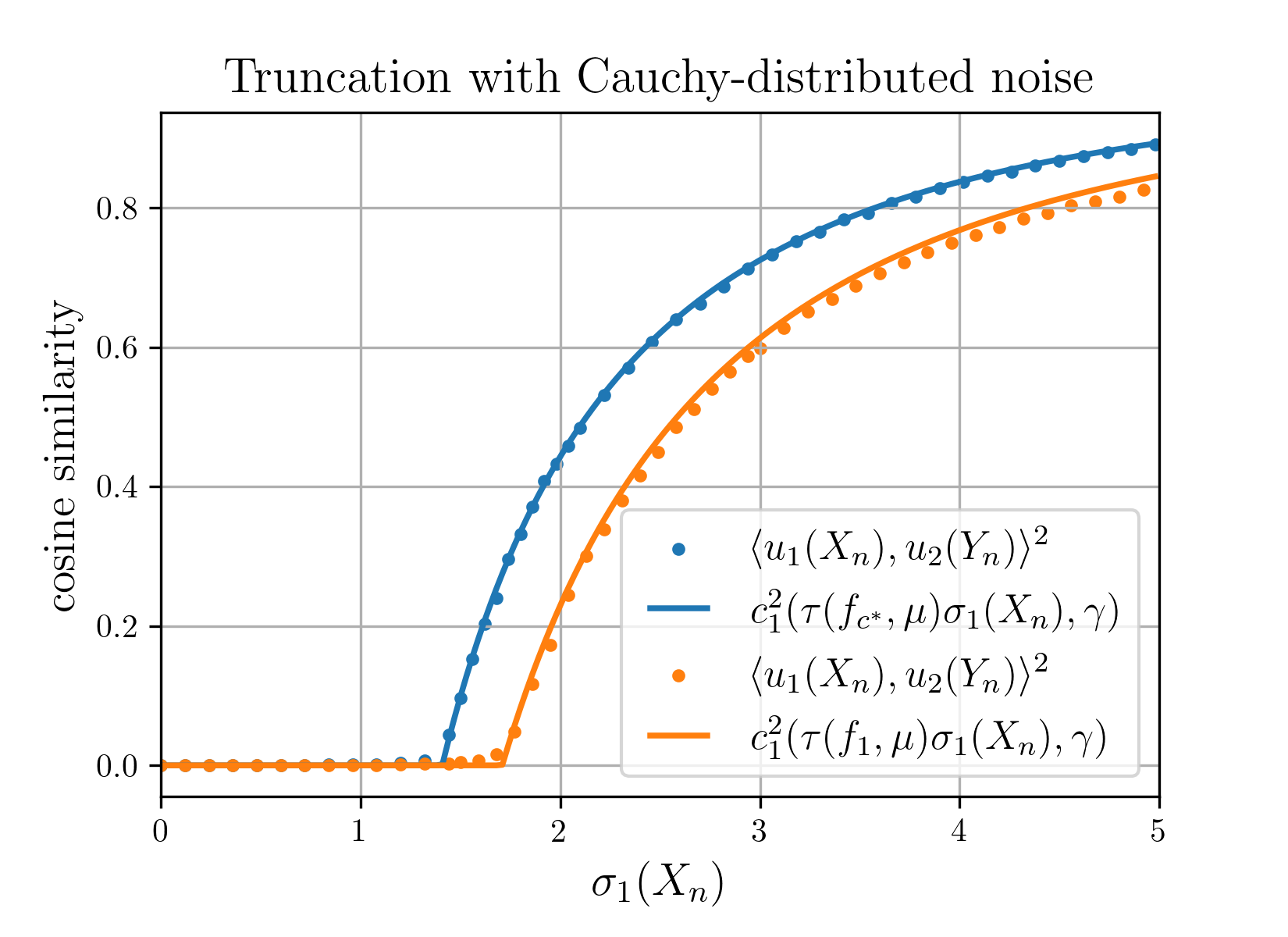} \hspace{-.6cm}
\includegraphics[height=2.4in]{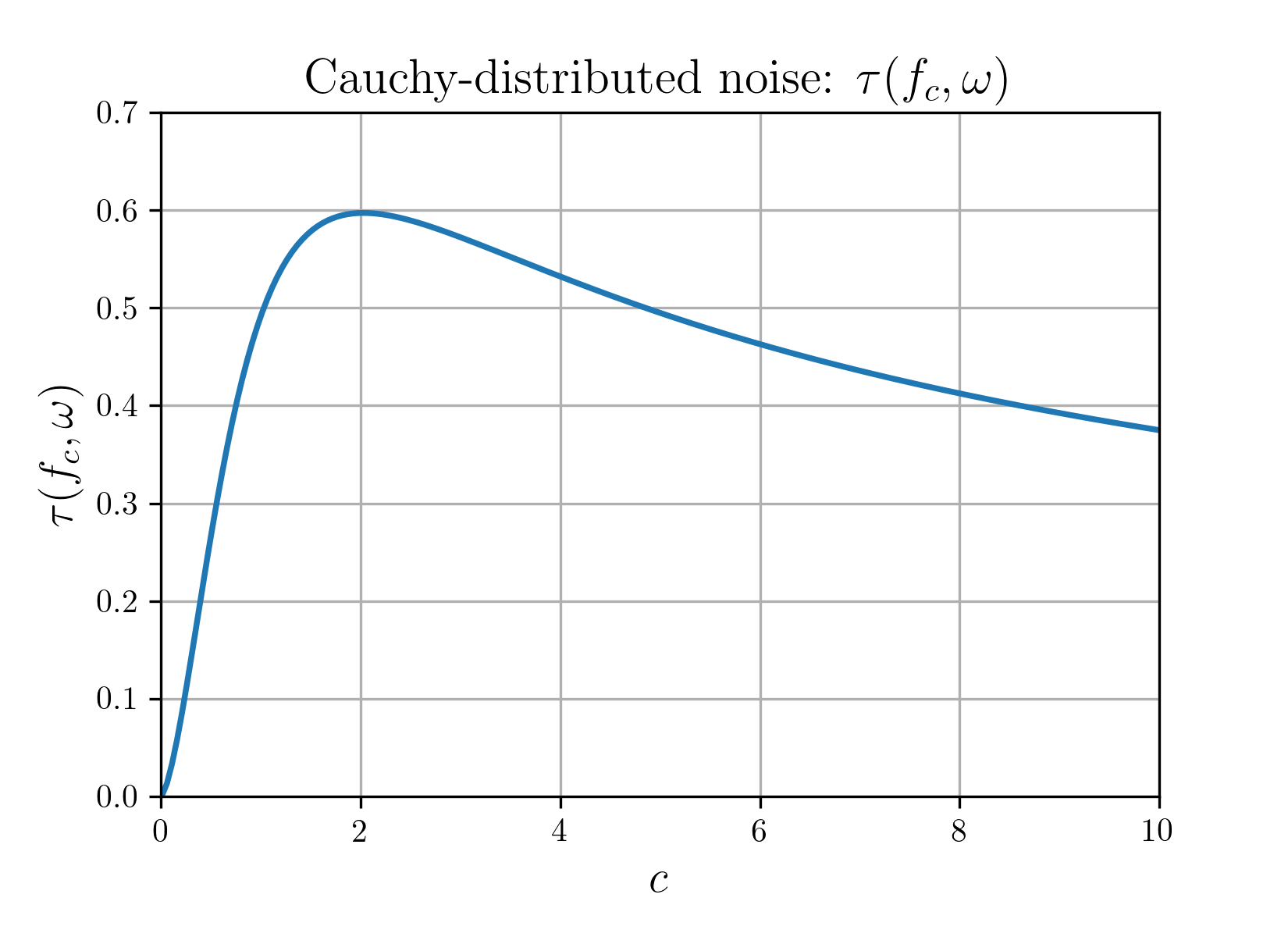}
\caption{ {\bf Left}: cosine similarities between the left singular vectors of $X_n$ and $Y_n = n^{-1/2} f_c(\sqrt{n}X_n+Z_n)$ with Cauchy-distributed noise and $c = c^*$ (blue) and $c =1$ (orange). Cosine similarities between the  singular vectors of $X_n$ and the raw data $X_n + n^{-1/2}Z_n$ are not plotted as they are $O(n^{-1/2})$ over the domain of this plot. There is close agreement between theory (solid lines) and simulations (points, each representing the average  25 simulations).
\\ {\bf Right}: Under Cauchy-distributed noise, $\tau(f_c,\mu)$ is maximized at $c^* \approx 2.028$.}
\label{fig3}
\end{figure}

\noindent {\bf Optimal singular value shrinkage.} Due to singular value bias and singular vector inconsistency, estimation of $X_n$ is improvable by singular value shrinkage. Assume the spiked matrix model $Y_n = X_n + n^{-1/2} Z_n$ (under the conditions of Lemma \ref{lem:bgn12}) and consider estimators of $X_n$ of the form \[
Y_{n, \eta} \coloneqq \sum_{i=1}^{n \wedge p } \eta(\sigma_i(Y_n)) u_i(Y_n) v_i(Y_n)^\top , \]
where $\eta: \mathbb{R}_{\geq 0} \mapsto  \mathbb{R}_{\geq 0}$ is a shrinkage rule. There exists an optimal rule $\eta^*$  developed in \cite{GD17, leeb}\footnote{\cite{leeb} corrects an error in the formula of $\eta^*$ for $\gamma \neq 1$ in \cite{GD17}.}, given by
\begin{equation} \label{shrink}
\begin{aligned}
    \eta^*(\sigma) &\coloneqq t(\sigma) \sqrt{\frac{t^2(\sigma) + \min(1,\gamma)}{t^2(\sigma) + \max(1,\gamma)}}, \\
 t^2(\sigma) & \coloneqq \begin{dcases} \frac{\sigma^2 - 1 - \gamma + \sqrt{(\sigma^2 - 1 - \gamma)^2 - 4 \gamma}}{2} &\sigma > 1 + \sqrt{\gamma} \\
0 &       \sigma \leq 1 + \sqrt{\gamma} 
 \end{dcases} ,
\end{aligned}    
\end{equation}
such that 
\begin{align}
     \lim_{n \rightarrow \infty}\| X_n - Y_{n, \eta^*}\|_2 =   \underset{{\alpha \in \mathbb{R}_{\geq 0}^r}}{\mathrm{argmin}} \lim_{n \rightarrow \infty} \Big\| X_n - \sum_{i=1}^r \alpha_i u_i(Y_n) v_i(Y_n)^\top \Big\|_2  
\end{align}
(the limits are well-defined almost surely). That is, applying $\eta^*$ to the (biased) singular values of $Y_n$ (without knowledge of the rank $r$ of $X_n$), $Y_{n, \eta^*}$ asymptotically achieves the optimal operator-norm loss for estimation of $X_n$ among all shrinkage estimators based on $Y_n$. 

Corollary \ref{cor:shrink} below states that $\eta^*$, designed for the spiked matrix model, extends (with no modifications necessary) to the elementwise-transformed model.  The proof follows from Theorem \ref{thrm1} and the proof of Theorem 1 of \cite{GD17}. In essence, this is true because $Y_n$ is approximately a spiked matrix with signal term proportional to $X_n$. For example, given binomial data of the form in (\ref{d3x}),  $Y_{n,\eta^*}$ yields an improved estimate of the structure $X_n$ compared to the rank-aware estimator $\sum_{i=1}^r \sigma_i(Y_n) u_i(Y_n) v_i(Y_n)^\top$.   

\begin{corollary}\label{cor:shrink}
    Let  $Y_n \coloneqq n^{-1/2} f(\sqrt{n} X_n + Z_n)$  and  $\|f\|_\mu = 1$. Under the asymmetric setting and assumptions (i)--(iv), $Y_{n,\eta^*}$ is an optimal shrinkage estimator of $\tau(f,\mu) X_n$ under operator norm loss: 
    \begin{align}
     \lim_{n \rightarrow \infty}\| \tau(f,\mu) X_n - Y_{n, \eta^*}\|_2 =   \underset{{\alpha \in \mathbb{R}_{\geq 0}^r}}{\mathrm{argmin}} \lim_{n \rightarrow \infty} \Big\| \tau(f,\mu)  X_n - \sum_{i=1}^r \alpha_i u_i(Y_n) v_i(Y_n)^\top \Big\|_2 .   
\end{align}
\end{corollary}

\begin{remark}
    We  assume $\|f\|_\mu =1$ in Corollary \ref{cor:shrink} without loss of generality: if $\|f\|_\mu \neq 1$, estimate $X_n$ by 
    \[ Y_{n,\eta} \coloneqq \|f\|_\mu \sum_{i=1}^{n \wedge p} \eta\big(\|f\|_\mu^{-1} \sigma_i(Y_n)\big) u_i(Y_n) v_i(Y_n)^\top .\]
     Note that $\|f\|_\mu$ is consistently estimated by $(1+\sqrt{\gamma})^{-1} \sigma_{k}(Y_n)$, where $k > r$ is any fixed upper bound on the rank of $X_n$ (see Lemma \ref{lem:max_eig0}).
\end{remark}



\begin{remark}
    An analogous results holds in the symmetric setting; optimal shrinkage functions for the spiked Wigner model are derived in \cite{DF22}.
\end{remark}


\section{Proof of  Theorems \ref{thrm1} and \ref{thrm2}} \label{sec:proofs}

In this section, we prove Theorems \ref{thrm1} and \ref{thrm2} in the asymmetric setting;  proofs in the symmetric setting are similar and omitted. Proofs of corollaries are deferred to the appendix. 

Our approach is to argue (1) $Y_n$ is well approximated by $n^{-1/2} f_K(\sqrt{n}X_n + Z_n)$ for large $K$ (Lemma \ref{lem1}) and (2) Theorems \ref{thrm1} and \ref{thrm2} hold under polynomial transformations (Lemma \ref{lem2}). We first recall a classic result in random matrix theory:

\begin{lemma} \label{lem:max_eig0} (Theorem 3.1 of \cite{YB88})
     Let $Z_n \coloneqq  (z_{ij}: 1 \leq i \leq n, 1 \leq j \leq p)$ denote an i.i.d.\ array of real random variables with mean zero, variance $\sigma^2$, and finite fourth moment. As $n \rightarrow \infty$ and $p/n \rightarrow \gamma > 0$,
     \begin{align}
         \frac{1}{\sqrt{n}} \|Z_n\|_2 \xrightarrow{a.s.} (1+\sqrt{\gamma})\sigma. 
     \end{align}
\end{lemma}

\noindent We also require the following extension of Lemma \ref{lem:max_eig0}, the proof of which is a straightforward generalization of the proof of equation (4.1) in \cite{YB88}: 
\begin{lemma} \label{lem:max_eig}
    Let $Z_n \coloneqq (z_{ij}^{(n)}: 1 \leq i \leq n, 1 \leq j \leq p)$ denote an array of independent, real random variables with mean zero and uniformly bounded  second and fourth moments:
    \begin{align*}
        & \limsup_{n \rightarrow \infty} \sup_{1 \leq i \leq n, 1 \leq j \leq p} \E | z_{ij}^{(n)} |^2 \leq \sigma^2 , && \limsup_{n \rightarrow \infty}  \sup_{1 \leq i \leq n, 1 \leq j \leq p} \E | z_{ij}^{(n)} |^4 < \infty.
    \end{align*}
As $n \rightarrow \infty$ and $p/n \rightarrow \gamma > 0$, almost surely, 
\begin{align} \label{kmr6}  \limsup_{n \rightarrow \infty} \frac{1}{\sqrt{n}} \|Z_n\|_2 \leq (1 + \sqrt{\gamma}) \sigma . 
\end{align}
\end{lemma}

\begin{lemma} \label{lem1} Let $\Delta_{n,K} \coloneqq Y_n - n^{-1/2} f_K(\sqrt{n}X_n + Z_n)$. Under the asymmetric setting and assumptions (i)--(iv),
    \begin{align}
        \lim_{K \rightarrow \infty} \lim_{n \rightarrow\infty} \|\Delta_{n,K}\|_2 \stackrel{a.s.}{=} 0 .  \label{as8jd}
    \end{align}
    
\end{lemma}
\begin{proof}  We shall use Lemma \ref{lem:max_eig} to bound the operator norm of $\Delta_{n,K} - \E\Delta_{n,K}$. First, we must establish bounds on the moments of elements of $\Delta_{n,k}$.   Since $f$ is polynomially bounded and continuous $\mu$-almost everywhere, the dominated convergence theorem yields
\begin{align*}
 & \lim_{x \rightarrow 0} \int (f(x+z) - f(z))^2 d \mu(z) = 0 , &&  \lim_{x\rightarrow 0} \int (f_K(x+z) - f_K(z))^2 d \mu(z) = 0 . 
\end{align*}
Together with  the bound   
\begin{align*}
    \int (f(x+z) - f_K(x+z))^2 d\mu(z) & \leq 3 \int (f(x+z) - f(z))^2 d\mu(z) + 3 \int (f(z) - f_K(z))^2 d\mu(z)  \\
    & + 3 \int (f_K(x) - f_K(x+z))^2 d\mu(z) ,
\end{align*}
these limits and (\ref{assumption4.0}) imply
\begin{align*}
    \lim_{K \rightarrow 0} \lim_{x \rightarrow 0}  \int (f(x+z) - f_K(x+z))^2 d\mu(z) = 0.
\end{align*}

Thus, as the elements of $\sqrt{n} X_n$ uniformly converge to zero (assumption (i)), the second moments of elements of $\sqrt{n}\Delta_{n,k}$ are uniformly controlled:
\begin{align} \label{42}
    \lim_{K \rightarrow \infty} \lim_{n \rightarrow \infty} \sup_{1 \leq i \leq n, 1 \leq j \leq p} \E\big|\sqrt{n} y_{ij} - f_K(\sqrt{n} x_{ij} + z_{ij})\big|^2 =  0 . 
\end{align}
Similarly, we have a uniform bound on the fourth moments of elements of $\sqrt{n} \Delta_{n,K}$: 
\begin{align*} 
     \lim_{n \rightarrow \infty} \sup_{1 \leq i \leq n, 1 \leq j \leq p} \E\big|\sqrt{n} y_{ij} - f_K(\sqrt{n} x_{ij} + z_{ij})\big|^4 < \infty . 
\end{align*}
Since the elements of $\sqrt{n} \Delta_{n,K}$ are independent, 
Lemma \ref{lem:max_eig} enables us to conclude that  
\begin{align} \label{as7jd}
\lim_{K \rightarrow \infty} \limsup_{n \rightarrow \infty} \| \Delta_{n,K} - \E \Delta_{n,K} \|_2 \stackrel{a.s.}{=} 0 .
\end{align}

It remains to bound $\| \E \Delta_{n,K}\|_2$. Using $\|\E \Delta_{n,K}\|_2 \leq \|\E \Delta_{n,K}\|_F$, we have 
\begin{equation}
\begin{aligned}  \label{as9jd}
 \|\E \Delta_{n,K} \|_2 & \leq  \frac{1}{\sqrt{n}} \bigg(\sum_{i=1}^n \sum_{j=1}^p \big [\E (\sqrt{n}y_{ij} - f_K(\sqrt{n}x_{ij} + z_{ij})) \big]^2 \bigg)^{1/2} \\
& \leq   \|X_n\|_F \cdot \sup_{\substack{1 \leq i \leq n, 1 \leq j \leq p, \\ x_{ij} \neq 0}} \bigg| \frac{1}{\sqrt{n} x_{ij}} \int (f(\sqrt{n} x_{ij} + z) - f_K(\sqrt{n} x_{ij} + z)) d \mu(z) \bigg|  
\end{aligned}    
\end{equation}   
(as $x_{ij} = 0$ implies $\E (\Delta_{n,k})_{ij} = 0$,  to bound the Frobenius norm of $\Delta_{n,K}$, we restrict attention to indices with $x_{ij} \neq 0$).
Therefore, using $\|X_n\|_F^2 = \sum_{i=1}^r \sigma_i^2(X_n) < \infty$, assumption (i), and  (\ref{assumption4.1}), 
\begin{align} \label{qerg}
    \lim_{K \rightarrow \infty} \limsup_{n \rightarrow \infty} \| \E \Delta_{n,K} \|_2 \stackrel{a.s.}{=} 0 .
\end{align}
Equations (\ref{as7jd}) and (\ref{qerg}) yield (\ref{as8jd}), completing the proof.
\end{proof}

\begin{lemma}\label{lem2} Define the matrix
\[
A_{n,K} \coloneqq \bigg(\sum_{k=1}^K a_k b_k \bigg)  X_n + \frac{1}{\sqrt{n}} f_K(Z_n) .
\]
Under the asymmetric setting and assumptions (i)--(iv),
\[
 \lim_{n \rightarrow \infty} \|Y_n - \Delta_{n,K} -  A_{n,K} \|_2 \stackrel{a.s.}{=} 0 . 
\]
\end{lemma}

\begin{proof}
We fix $K$ and Taylor expand $f_K(\sqrt{n}x_{ij} + z_{ij})$:
\begin{align*}
    f_K(\sqrt{n} x_{ij} + z_{ij})& = \sum_{k=1}^K a_k q_k(\sqrt{n} x_{ij} + z_{ij})  =  \sum_{k=1}^K  \sum_{\ell = 0}^k \frac{a_k}{\ell!} (\sqrt{n} x_{ij})^\ell q_k^{(\ell)}(z_{ij}).   
\end{align*}
Equivalently, 
\begin{align} \label{as10jd}
     f_K(\sqrt{n} X_n + Z_n) &=  \sum_{k=1}^K \sum_{\ell=0}^k \frac{a_k}{\ell!} (\sqrt{n} X_n)^{\odot \ell} \odot q_k^{(\ell)}(Z_n). 
\end{align}

Consider terms  with $\ell \geq 1$. As $q_k^{(\ell)}$ is a polynomial of degree $k-\ell$, $q_k^{(\ell)}(Z_n)$ is an i.i.d.\ array of variables with mean $\langle q_k^{(\ell)}, 1 \rangle_\mu$, variance $\|q_k^{(\ell)}\|_\mu^2 - \langle q_k^{(\ell)}, 1 \rangle_\mu^2$, and finite moments. Therefore,  the Hadamard product $(\sqrt{n} X_n)^{\odot \ell} \odot \big(q_k^{(\ell)}(Z_n) - \langle q_k^{(\ell)}, 1 \rangle_\mu {\bf 1}_n {\bf 1}_p^\top \big)$ has independent elements with mean zero and fourth moments uniformly converging to zero (by assumption (i)). Hence, by Lemma \ref{lem:max_eig},
\begin{gather}
 \frac{1}{\sqrt{n}} \big\| (\sqrt{n} X_n)^{\odot \ell} \odot \big(q_k^{(\ell)}(Z_n) - \langle q_k^{(\ell)}, 1 \rangle_\mu {\bf 1}_n {\bf 1}_p^\top \big) \big\|_2 \xrightarrow{a.s.} 0 . \label{47} \end{gather}
 Thus, terms of (\ref{as10jd}) indexed by $\ell \geq 1$ are dominated in operator norm by their expectations: 
 \begin{gather*}
  \frac{1}{\sqrt{n}} \Big\| f_K(\sqrt{n}X_n+Z_n) -\sum_{k=1}^K a_k \Big( q_k(Z_n) + \sum_{\ell=1}^k \frac{1}{\ell!} \langle q_k^{(\ell)} ,1 \rangle_\mu(\sqrt{n}X_n)^{\odot \ell} \Big) \Big\|_2 \xrightarrow[]{a.s.} 0.
\end{gather*}

Together with assumption (i), which implies that $n^{-1/2} \| (\sqrt{n} X_n)^{\odot \ell} \|_2 \rightarrow 0$ for $\ell \geq 2$, we find that terms of  (\ref{as10jd}) with $\ell \geq 2$  
are negligible:
\begin{align} \|Y_n - \Delta_{n,K} - A_{n,K}\|_2 =  
    \frac{1}{\sqrt{n}} \Big\|f_K(\sqrt{n} X_n + Z_n) - \sum_{k=1}^K a_k \big( q_k(Z_n) +   \langle q_k',1 \rangle_\mu  \sqrt{n} X_n \big)   \Big\|_2 &  \xrightarrow{a.s.} 0 , \label{48}
\end{align}
completing the proof. 


\end{proof}

\begin{proof}[Proof of Theorem \ref{thrm1}] 
By Lemmas \ref{lem1} and \ref{lem2},
    \begin{align} \label{49}
         \lim_{K \rightarrow \infty} \limsup_{n \rightarrow \infty} \|Y_n - A_{n,K}\|_2 \leq \lim_{K \rightarrow \infty} \limsup_{n \rightarrow \infty} \big( \|Y_n - \Delta_{n,K} - A_{n,k}\|_2 +  \|\Delta_{n,K}\|_2 \big) \stackrel{a.s.}{=} 0 .
    \end{align}
Furthermore, we have
\begin{align}
    \|A_n - A_{n,K}\|_2 \leq \Big|\tau(f,\mu)\|f\|_\mu - \sum_{k=1}^K a_k b_k \Big| \cdot \|X_n\|_2 + \frac{1}{\sqrt{n}} \|f(Z_n) - f_K(Z_n)\|_2 .
\end{align}
Using Lemmas \ref{lem:tau} and \ref{lem:max_eig0}, the first term on the right-hand side of the above equation vanishes as $K \rightarrow \infty$, while the second term satisfies 
\[ \frac{1}{\sqrt{n}} \|f(Z_n) - f_K(Z_n)\|_2 \xrightarrow{a.s.} (1+\sqrt{\gamma}) (\|f\|_\mu^2 - \|f_K\|_\mu^2)^{1/2}. \]
Since $\|f\|_\mu^2 - \|f_K\|_\mu^2 \rightarrow 0$ as $K \rightarrow \infty$ by (\ref{assumption4.0}), we obtain
\begin{align}
    \lim_{K \rightarrow \infty} \limsup_{n \rightarrow \infty} \|A_n - A_{n,K}\|_2 \stackrel{a.s.}{=} 0 .
\end{align}
Thus, 
\begin{align} \label{412}
    \lim_{n \rightarrow \infty} \|Y_n - A_n\|_2 \leq  \lim_{K \rightarrow \infty} \limsup_{n \rightarrow \infty}  \Big( \|Y_n - A_{n,K}\|_2 + \|A_n - A_{n,K}\|_2 \Big) \stackrel{a.s.}{=} 0 .
\end{align}

\end{proof}

\begin{proof}[Proof of Theorem \ref{thrm2}] The proof is a refinement of that of Theorem \ref{thrm1}. In light of the proof of Lemma \ref{lem1}, assumption (ii$\mathrm{'}$), and Lemma \ref{lem:assumption1}, introducing the matrix
\[\Delta_{n,K} \coloneqq Y_n - \frac{1}{\sqrt{n}} f_K(n^{1-1/(2\ell_*)}X_n + Z_n) ,\]
we have
 \begin{align} \label{426}
        \lim_{K \rightarrow \infty} \limsup_{n \rightarrow\infty} \|\Delta_{n,K}\|_2 \stackrel{a.s.}{=} 0 .
    \end{align}

Fixing $K$ and Taylor expanding $f_K(n^{1-1/(2\ell_*)}X_n + Z_n)$ analogously to (\ref{as10jd}),
\begin{align} \label{427}
    f_K(n^{1-1/(2\ell_*)}X_n + Z_n) = \sum_{k=1}^K \sum_{\ell =0}^k \frac{a_k}{\ell!} (n^{1-1/(2\ell_*)} X_n)^{\odot \ell} \odot q_k^{(\ell)}(Z_n) . 
\end{align}
Similarly to (\ref{47}), assumption (i$\mathrm{'}$) and Lemma \ref{lem:max_eig} imply that terms of (\ref{427}) indexed by $\ell \geq 1$ are dominated in operator norm by their expectations:
\begin{align} \label{428}
 \frac{1}{\sqrt{n}} \big\| (n^{1-1/(2\ell_*)} X_n)^{\odot \ell} \odot \big(q_k^{(\ell)}(Z_n) - b_{k\ell} {\bf 1}_n {\bf 1}_p^\top \big) \big\|_2 \xrightarrow{a.s.} 0  .
\end{align}

Since $X_n^{\odot \ell} = \sigma_1^\ell(X_n) ((u_1(X_n))^{\odot \ell})^\top (v_1(X_n))^{\odot \ell}$ is rank one, 
\begin{equation*}
    \begin{aligned}
    \|(n^{1-1/(2\ell_*)} X_n)^{\odot \ell}\|_2 &= n^{\ell - \ell/(2\ell_*)} \sigma_1^\ell(X_n) \|(u_1(X_n))^{\odot \ell}\|_2 \|(v_1(X_n))^{\odot \ell}\|_2 \\
& = n^{(1+\ell - \ell/\ell_*)/2} p^{(1-\ell)/2} \sigma_1^\ell(X_n) \cdot \frac{1}{\sqrt{n}} \|\sqrt{n} u_1(X_n)\|_{2\ell}^\ell \cdot \frac{1}{\sqrt{p}} \|\sqrt{p} v_1(X_n)\|_{2\ell}^\ell .
\end{aligned}
\end{equation*}
Thus, using assumption (i$\mathrm{'}$) and $p/n \rightarrow \gamma$, we have $n^{-1/2}\|(n^{1-1/(2\ell_*)} X_n)^{\odot \ell}\|_2 \rightarrow 0$ for $\ell > \ell_*$ and
\begin{align}
    \frac{1}{\sqrt{n}}\big\|(n^{1-1/(2\ell_*)} X_n)^{\odot \ell_*}\big\|_2 \rightarrow \gamma^{(1-\ell_*)/2} \sigma_1^{\ell_*}(X_n) \sqrt{m_{2\ell_*}^u  m_{2\ell_*}^v} . \label{430} 
\end{align}
Together with (\ref{428}), this yields the approximation
\begin{align} \label{queermo}
    \frac{1}{\sqrt{n}} \Big\| f_K(n^{1-1/(2\ell_*)}X_n + Z_n) - \sum_{k=1}^K \sum_{\ell = 1}^{k \wedge \ell_*} \frac{a_k b_{k\ell}}{\ell!} (n^{1-1/(2\ell_*)}X_n)^{\odot \ell} - f_K(Z_n) \Big\|_2 \xrightarrow{a.s.} 0 .
\end{align}

Rewriting the second term of the above equation,
\begin{align*}
    \sum_{k=1}^K \sum_{\ell = 1}^{k \wedge \ell_*} \frac{a_k b_{k\ell}}{\ell!} (n^{1-1/(2\ell_*)}X_n)^{\odot \ell} = \sum_{\ell=1}^{\ell_*} \sum_{k = \ell}^{K} \frac{a_k b_{k\ell}}{\ell!} (n^{1-1/(2\ell_*)}X_n)^{\odot \ell} ,
\end{align*}
recall that  $\sum_{k=\ell }^K a_k b_{k\ell} = 0$ for $\ell < \ell_*$ and sufficiently large $K$ by assumption (iv$\mathrm{'}$).\footnote{To avoid assumption  (iv$\mathrm{'}$), for $\ell < \ell_*$, we would need a bound on the convergence rate of $\sum_{k=\ell}^K a_k b_{k\ell}$ to zero and to argue that there exists $K_n \rightarrow \infty$ such that $\sum_{k=1}^{K_n} a_k b_{k\ell} \cdot \|(n^{1-1/(2\ell_*)} X_n)^{\odot \ell}\|_2 \rightarrow 0$. Analogs of (\ref{426})--(\ref{queermo}) that permit the degree $K_n$ of the polynomial approximation to diverge would also be required. We leave this analysis for future work. } Thus, 
\begin{gather} \label{poiu}
\frac{1}{\sqrt{n}} \Big\| f_K(n^{1-1/(2\ell_*)}X_n + Z_n) - \sum_{k=1}^K \frac{a_k   b_{k\ell_*}}{\ell_*!} (n^{1-1/(2\ell_*)}X_n)^{\odot \ell_*} - f_K(Z_n) \Big\|_2 \xrightarrow{a.s.} 0 .
\end{gather}
From (\ref{426}), (\ref{430}), (\ref{poiu}), and calculations analogous to (\ref{49})--(\ref{412}), we obtain 
\begin{gather}
 \lim_{n \rightarrow \infty} \|Y_n- A_{\ell_*,n}\|_2 \xrightarrow[]{a.s.} 0 ,
\end{gather}
completing the proof.

\end{proof}

\appendix

\section{Appendix} \label{appendix}

\subsection{Auxiliary Lemmas}
\begin{lemma} \label{lem:tau}
    Under assumptions (ii)--(iv), $\sum_{k=1}^\infty a_k b_k$ converges. 
\end{lemma}
\begin{proof} By (\ref{assumption4.1}),
    \begin{align*}
        \lim_{K \rightarrow \infty} \lim_{L \rightarrow \infty} \lim_{x \rightarrow 0}  \bigg|  \frac{1}{x} \int (f_K(x+z) - f_L(x+z)) d \mu(z)  \bigg|  \leq 2\lim_{K \rightarrow \infty} \lim_{x \rightarrow 0}  \bigg| 
 \frac{1}{x}  \int (f(x+z) - f_K(x+z)) d\mu(z) \bigg|= 0 . 
    \end{align*}
On the other hand, 
\begin{align}
    \nonumber &~ \lim_{x \rightarrow 0} \frac{1}{x} \int (f_K(x+z) - f_L(x+z)) d \mu(z) \\  = &~   \lim_{x \rightarrow 0} \frac{1}{x} \int \big(f_K(x+z) - f_K(z) - f_L(x+z) + f_L(z)\big) d \mu(z)  \\
    = &~     \int (f_K'(z) - f_L'(z)) d \mu(z) = \sum_{k = K \wedge L}^{K \vee L}  a_k b_k ,  \nonumber
\end{align}
where  the first equality holds as $\E q_k(z_{11}) = 0$ for $k > 1$ and the second holds by the dominated convergence theorem. Thus,
\begin{align}
     \lim_{K \rightarrow \infty} \sum_{k=K}^\infty a_k b_k = 0 ,
\end{align}
completing the proof.
\end{proof}

\begin{lemma}\label{lem:tau_f'}
    Let assumptions (ii)--(iv) hold, $f$ be differentiable, $\mu$ have differentiable density $\omega$,  and $\omega'/\omega$ be polynomially bounded. Then, $\tau(f,\mu)= \|f\|_\mu^{-1} \E f'(z_{11})$. 
\end{lemma}
\begin{proof}
Using integration by parts and the Cauchy-Schwarz inequality,
\begin{equation}
\begin{aligned}
 \bigg| \int_{-\infty}^\infty (f'(z) - f_K'(z)) \omega(z) dz \bigg|  &=    \bigg| \int_{-\infty}^\infty (f(z) - f_K(z)) \omega'(z) dz \bigg| \\
    & \leq \bigg( \int_{-\infty}^\infty (f(z) - f_K(z))^2 \omega(z) dz \bigg)^{1/2} \bigg( \int_{-\infty}^\infty \frac{(\omega'(z))^2}{\omega(z)} dz \bigg)^{1/2} .  
\end{aligned}
\end{equation}
The right-hand side vanishes as $K \rightarrow \infty$ by (\ref{assumption4.0}) and assumption (ii), implying 
\[\tau(f,\mu) = \frac{1}{\|f\|_\mu}\lim_{K \rightarrow \infty} \E f_K'(z_{11}) =  \frac{1}{\|f\|_\mu} \E f'(z_{11}) . \]
\end{proof}

\begin{lemma} \label{lem:assumption1}
   If  $\mu$ has differentiable density $\omega$, $\text{supp}(\omega) = \mathbb{R}$, and $\omega'/\omega$ polynomially bounded, assumptions (iii) and (\ref{assumption4.0}) imply (\ref{assumption4.1}):
           \begin{align*}
        \lim_{K \rightarrow \infty} \lim_{x \rightarrow 0}    \frac{1}{x} \int ( f(x+z) - f_K(x+z)) \omega(z) dz= 0.    
    \end{align*}
\end{lemma}

  \begin{proof}
    Using $\E(f(z_{11}) - f_K(z_{11})) = 0$ and the Cauchy-Schwarz inequality, we have
\begin{align}
    \int_{-\infty}^\infty (f(x+z) - f_K(x+z)) \omega(z) dz &=    \int_{-\infty}^\infty (f(z) - f_K(z)) \omega(-x+z)dz \nonumber  \\
    &=     \int_{-\infty}^\infty (f(z) - f_K(z)) (\omega(-x+z)-\omega(z)) dz \\
    & \leq \bigg(\int_{-\infty}^\infty (f(z) - f_K(z))^2 \omega(z) dz \bigg)^{1/2} \bigg( \int_{-\infty}^\infty \frac{( \omega(-x+z) - \omega(z))^2}{\omega(z)} dz\bigg)^{1/2}.
\nonumber \end{align}
By the mean value theorem, the polynomial boundedness of $\omega'/\omega$, and the dominated convergence theorem, 
\begin{align} 
 \lim_{x \rightarrow 0}  \int_{-\infty}^\infty \frac{( \omega(-x+z) - \omega(z))^2}{x^2 \omega(z)} dz = \int_{-\infty}^\infty \frac{(\omega'(z))^2}{\omega(z)} dz < \infty . 
\end{align}

Thus,
\begin{align}
\lim_{x \rightarrow 0} \bigg| \frac{1}{x} \int_{-\infty}^\infty (f(x+z) - f_K(x+z)) \omega(z) dz \bigg| \leq \bigg(\sum_{k=K+1}^\infty a_k^2 \bigg)^{1/2} \bigg( \int_{-\infty}^\infty \frac{(\omega'(z))^2}{\omega(z)} dz \bigg)^{1/2} .
\end{align}
The claim now follows from (\ref{assumption4.0}).  
  \end{proof}

\begin{lemma}\label{lem:tau2}
    Under assumptions (ii$\mathrm{'}$) and (iii$\mathrm{'}$), $|\tau_\ell(f,\mu)| < \infty$ for $\ell \in \mathbb{N}$. 
\end{lemma}
\begin{proof}
Since $\omega$ has finite moments, $z^k \omega(z) \rightarrow 0$ as $|z| \rightarrow \infty$. Using integration by parts,
\begin{align}
   b_{k\ell} = \int_{-\infty}^\infty  q_k^{(\ell)}(z) \omega(z) dz = (-1)^\ell \int_{-\infty}^\infty q_k(z) \omega^{(\ell)}(z) dz .
\end{align}
Thus, $b_{k\ell}$ is the projection of the function $(-1)^\ell \omega^{(\ell)}/\omega$ onto $q_k$.
Since this function is polynomially bounded, we have $\omega^{(\ell)}/\omega \in L^2(\mathbb{R},\mu)$, and Lemma \ref{lem:ortho2} therefore implies 
\begin{align}
    \sum_{k=1}^\infty b_{k\ell}^2 = \int_{-\infty}^\infty \frac{(\omega^{(\ell)}(z))^2}{\omega(z)} dz < \infty . 
\end{align}
The claim now follows from the Cauchy-Schwarz inequality: 
\[ |\tau_\ell(f,\mu)| = \frac{1}{\|f\|_\mu} \bigg| \sum_{k=1}^\infty a_k b_{k\ell} \bigg| \leq  \frac{1}{\|f\|_\mu} \bigg( \sum_{k=1}^\infty a_k^2 \cdot \sum_{k=1}^\infty b_{k\ell}^2 \bigg)^{1/2}  =  \bigg( \sum_{k=1}^\infty b_{k\ell}^2 \bigg)^{1/2}.
\]
\end{proof}

\subsection{Proof of Corollaries}

We first prove Corollary \ref{cor1}; the proof of Corollary \ref{cor2} is similar and omitted.  
\begin{proof}[Proof of (\ref{g27v})]

 As $f(Z_n)$ is an i.i.d.\ array of variables with mean zero, variance $\|f\|_\mu^2$, and finite moments ($f$ is polynomially bounded), the ESD of $\|f\|_\mu^{-2} A_n^\top A_n$ converges by Lemma \ref{lem:bgn12} almost surely weakly to the Marchenko--Pastur law with parameter $\gamma$. The LSD of $\|f\|_\mu^{-2} Y_n^\top Y_n$ is identical by Lemma 2.1 of \cite{ElK}.
 
 Equations (\ref{qsd1}) and (\ref{bgn12_1}) yield expressions for the limiting singular value bias and singular vector inconsistency of $A_n$:
\begin{align} \label{g29v}
&&  \hspace{3.8cm}\sigma_i^2(A_{n}) \xrightarrow[]{a.s.} \lambda(\tau(f,\mu) \sigma_i(X_n),\gamma), \hspace{.25cm} && \hspace{.9cm} 1 \leq i \leq r, 
\end{align}
\begin{equation}
 \begin{aligned}  \label{g30v} \hspace{3.75cm}
    \langle u_i(X_n), u_j(A_n) \rangle^2 & \xrightarrow{a.s.} \delta_{ij} \cdot c_1^2 (\tau(f,\mu) \sigma_i(X_n), \gamma) , \\  \langle v_i(X_n), v_j(A_n) \rangle^2 & \xrightarrow{a.s.} \delta_{ij} \cdot c_2^2 (\tau(f,\mu)\sigma_i(X_n), \gamma) ,  
\end{aligned}   \hspace{1.45cm} 1\leq i,j \leq r .    
\end{equation}
By Weyl's inequality and Theorem \ref{thrm1},
\begin{align} \label{g31v}
     |\sigma_i(Y_n) - \sigma_i(A_n)| \leq  
 \|Y_n - A_{n} \|_2 \xrightarrow{a.s.} 0 .
\end{align}
Equation (\ref{g27v}) follows from (\ref{g31v}) and the bound
\begin{align*}
|\sigma_i^2(Y_n) - \sigma_i^2(A_n) | & = |\sigma_i(Y_n) - \sigma_i(A_n)|  | \sigma_i(Y_n) + \sigma_i(A_n)| \\
 &\leq |\sigma_i(Y_n) - \sigma_i(A_n)| (2 |\sigma_i(A_n)| +  |\sigma_i(Y_n) - \sigma_i(A_n)| ) .
\end{align*}
\end{proof}

To complete the proof of Corollary \ref{cor1}, it remains to prove (\ref{g28v}). We note that for supercritical singular values, $\tau(f,\mu) \sigma_i(X_n) > \gamma^{1/4}$, the convergence $\langle v_i(X_n), v_i(Y_n) \rangle^2 \xrightarrow{a.s.} c_2^2(\tau(f,\mu) \sigma_i(X_n), \gamma)$ is an immediate consequence of (\ref{g27v}) and the Davis-Kahan theorem (Theorem 4 of \cite{YWS14}).  However, for subcritical singular values,  $\tau(f,\mu) \sigma_i(X_n) \leq \gamma^{1/4}$, asymptotic cosine similarities cannot be similarly derived: as $\sigma_i^2(A_n) \xrightarrow{a.s.}(1+\sqrt{\gamma})^2$, the upper bulk edge of the Marchenko--Pastur law, $\sigma_i^2(A_n)$ does not satisfy the eigenvalue separation condition of the Davis-Kahan theorem.  

Define the Stieltjes transform $m_\gamma(z)$ of the  Marchenko--Pastur law with parameter $\gamma$:
\begin{align*} &&
 m_\gamma(z) \coloneqq - \frac{z - 1+ \gamma - \sqrt{(z - 1 - \gamma)^2 - 4 \gamma}}{2 \gamma z}  , &&  \gamma \in (0,1],
\end{align*}
and $m_{\gamma^{-1}}(z) = \gamma m_\gamma(z) - (1-\gamma) z^{-1}$.

\begin{lemma} \label{lem:zxcv} For $A \in \mathbb{R}^{n \times p}$ and $z \in \mathbb{C}^+$,
    \[ A(zI_p - A^\top A)^{-1} A^\top = -I_n + z (z I_n-A A^\top)^{-1} . \]
\end{lemma}
\begin{proof} This is a particular case of the Woodbury identity. 
\end{proof}

\begin{lemma} \label{lema5}
Let $Z_n \coloneqq (z_{ij} : 1 \leq i \leq n, 1 \leq j \leq p)$ denote an array of i.i.d.\ real random variables with mean zero, variance one, and finite moments. Let $S_n \coloneqq n^{-1} Z_n^\top Z_n$ 
and $V_n$ and $W_n$ be deterministic semi-orthogonal matrices of size $p \times r$ and $p \times (p-r)$, respectively, satisfying $V_n^\top W_n = 0$.
As $n \rightarrow \infty$ and $p/n \rightarrow \gamma \in (0,1]$,
\begin{align}
     && V_n^\top S_n W_n (zI_{p-r} - W_n^\top S_n W_n)^{-1} W_n^\top S_n V_n \xrightarrow{a.s.} (-z \gamma m_\gamma(z) + 1 - \gamma) I_r, && \,\, z \in \mathbb{C}^+ .  \label{a5}
\end{align}

\end{lemma}
\begin{proof} 
Using  the identity $V_n V_n^\top + W_n W_n^\top = I_p$ and the Woodbury formula,
\begin{equation} \label{a8q}
    \begin{aligned}\Big(\frac{1}{n} Z_n W_n W_n^\top Z_n^\top -  z I_n \Big)^{-1} & = 
     \Big( \frac{1}{n} Z_n (I_p - V_n V_n^\top) Z_n^\top - z I_n\Big)^{-1} \\
      & =  G_n(z) + \frac{1}{n} G_n(z) Z_n V_n \Big( I_r - \frac{1}{n} V_n^\top Z_n^\top G_n(z) Z_n V_n \Big)^{-1} V_n^\top Z_n^\top G_n(z) ,  
\end{aligned}
\end{equation}
where $G_n(z) \coloneqq (n^{-1} Z_n Z_n^\top - z I_n)^{-1}$ is the resolvent of the {\it companion matrix} of $S_n$. Thus, using Lemma~\ref{lem:zxcv} and denoting $A_n(z) \coloneqq n^{-1} V_n^\top Z_n^\top G_n(z) Z_n V_n$, we have
\begin{equation}
\begin{aligned} \label{a8}
    &~ V_n^\top S_n W_n (zI_{p-r} - W_n^\top S_n W_n)^{-1} W_n^\top S_n V_n \\ = &~ -\frac{1}{n} V_n^\top Z_n^\top \Big(I_n + z \Big( \frac{1}{n} Z_n W_n W_n^\top Z_n^\top -z I_n\Big)^{-1} \Big) Z_n V_n  \\
     = &~  -z A_n(z) ( I_r+ (I_r - A_n(z))^{-1} A_n(z)) - V_n^\top S_n V_n \\
     = &~ - z A_n(z)(I_r - A_n(z))^{-1} - V_n^\top S_n V_n .
\end{aligned}
\end{equation}
A second application of Lemma \ref{lem:zxcv} yields $A_n(z) = I_r + z V_n^\top (S_n - z I_p)^{-1} V_n$,  which converges to a multiple of identity by the isotropic local law of \cite{Bloe}: $A_n(z) \xrightarrow{a.s.} (1 + z m_\gamma(z))I_r$. Thus,
\begin{align} \label{a10}
    - z A_n(z)(I_r - A_n(z))^{-1} \xrightarrow{a.s.}  \Big( \frac{1}{m_\gamma(z)} + z \Big) I_r . 
\end{align}

Let $\alpha_n \in \mathbb{R}^p$ denote a deterministic sequence of unit vectors and $F_\gamma$ the Marchenko--Pastur law with parameter $\gamma$. 
Defining the empirical measure 
\[ \mu_n \coloneqq  \sum_{i=1}^p \langle \alpha_n, v_i(Z_n) \rangle^2 \delta_{\lambda_i(S_n)} ,
\] the isotropic local law states the Stieltjes transform of $\mu_n$ converges almost surely to $m_\gamma(z)$:
\begin{align}
     \int \frac{1}{\lambda - z} d\mu_n(\lambda) = \alpha_n^\top (S_n - zI_p)^{-1} \alpha_n \xrightarrow{a.s.} m_\gamma(z) ,
\end{align}
implying almost-sure weak convergence of $\mu_n$ to $F_\gamma$. Consequently,
\begin{align} \label{a17s}
     \alpha_n^\top S_n \alpha_n = \int \lambda \, d \mu_n(\lambda) \xrightarrow{a.s.} \int \lambda  \, dF_\gamma(\lambda) = 1 .
\end{align}
Given a deterministic sequence of unit vectors $\beta_n$ orthogonal to $\alpha_n$, 
a similar argument yields $\alpha_n^\top S_n \beta_n \xrightarrow{a.s.}0$, implying $V_n^\top S_n V_n \xrightarrow[]{a.s.} I_r$. Thus,  (\ref{a5})  follows from (\ref{a8}), (\ref{a10}), and the identity 
\[ \frac{1}{m_\gamma(z)} + z = -z \gamma m_\gamma(z) + 1 - \gamma . \]

\end{proof}

\begin{proof}[Proof (\ref{g28v})] 
We assume without loss of generality that $\|f\|_\mu = 1$ and (transposing $Y_n$ if necessary) $\gamma \in (0,1]$.
Denoting $\Delta_n \coloneqq Y_n - A_n$, we may write
\[
 Y_n = \tau(f,\mu) X_n + \frac{1}{\sqrt{n}} f(Z_n) + \Delta_n .
\]
Let $\tau(f,\mu) X_n = U_{n} \Lambda V_{n}^\top$, where $\Lambda \coloneqq \tau(f,\mu) \cdot \text{diag}( \sigma_1(X_n), \ldots \sigma_r(X_n))$ and $U_n$ and $V_n$ are semi-orthogonal matrices of size $n \times r$ and $p \times r$, respectively. Additionally, let $W_n$ denote a semi-orthogonal matrix of size $p \times (p-r)$  with columns spanning the orthogonal complement of $V_n$.

As in Section 3 of \cite{MJMY201}, introducing the matrix
\begin{align*}
    K_n(z) & \coloneqq  V_n^\top Y_n^\top Y_n W_n (z I_{p-r} -   W_n^\top Y_n^\top Y_n W_n)^{-1} W_n^\top Y_n^\top Y_n V_n ,\end{align*} 
    we have  $\big(V_n^\top Y_n^\top Y_n V_n + K_n(\sigma_i^2(Y_n)) \big) V_n^\top v_i(Y_n) = \sigma_i^2(Y_n) V_n^\top v_i(Y_n)$
    (note that $\sigma_i^2(Y_n)I_{p-r} - W_n^\top Y_n^\top Y_n W_n$ is invertible almost surely eventually). 
    Furthermore,  \begin{align} \label{b6s}
    v_i(Y_n)^\top V_n \big(I_r + \partial_z K_n(\sigma_i^2(Y_n)) \big) V_n^\top v_i(Y_n) = 1 ,
\end{align}
where $\partial_z K_n(z)$ is the elementwise derivative of $K_n(z)$: 
    \begin{align*}
     \partial_z K_n(z) & =  V_n^\top Y_n^\top Y_n W_n (z I_{p-r} -   W_n^\top Y_n^\top Y_n W_n)^{-2} W_n^\top Y_n^\top Y_n V_n .
\end{align*}

We shall first prove that for subcritical singular values, $\tau(f,\mu) \sigma_i(X_n) \leq \gamma^{1/4}$,  
\begin{align} \label{z6c} \sigma_r\big( \partial_z K_n(\sigma_i^2(Y_n))\big) \xrightarrow{a.s.} \infty, \end{align}
implying $\|V_n^\top v_i(Y_n)\|_2 \xrightarrow[]{a.s.} 0$ by (\ref{b6s}). 
Define the related matrix
\begin{align*}
    \overline K_n(z) &\coloneqq   \alpha_n^\top \Big(z I_{p-r} -   \frac{1}{n} W_n^\top f(Z_n)^\top f(Z_n) W_n\Big)^{-1} \alpha_n , \\
    \alpha_n &\coloneqq \frac{1}{\sqrt{n}} W_n^\top f(Z_n)^\top  \Big(U_n \Lambda + \frac{1}{\sqrt{n}} f(Z_n) V_n\Big) ,
\end{align*}
and let $G_n(z) \coloneqq (zI_{p-r} - W_n^\top Y_n^\top Y_n W_n)^{-1}$.
Since $Y_n W_n = (n^{-1/2} f(Z_n) + \Delta_n)W_n$, we have
\begin{equation*}
\begin{aligned}
    \big\|K_n(z) - \overline K_n(z)\big\|_2 & \leq   \Big\|\alpha_n^\top \Big( G_n(z) - \Big( zI_{p-r}- \frac{1}{n} W_n^\top f(Z_n)^\top f(Z_n) W_n\Big)^{-1} \Big) \alpha_n \Big\|_2\\
    &  + 2 \| \alpha_n^\top G_n(z) W_n^\top \Delta_n^\top Y_n V_n\|_2 +   \|V_n^\top Y_n^\top \Delta_n W_n G_n(z) W_n^\top \Delta_n^\top Y_n V_n\|_2 \\
    &  + \frac{2}{\sqrt{n}} \|\alpha_n^\top G_n(z) W_n^\top f(Z_n)^\top \Delta_n V_n\|_2 + \frac{1}{n} \| V_n^\top \Delta_n^\top f(Z_n) W_n G_n(z) W_n^\top f(Z_n)^\top \Delta_n V_n  \|_2 \\
    & + \frac{2}{\sqrt{n}} \|W_n^\top \Delta_n^\top Y_n V_n G_n(z) W_n^\top f(Z_n)^\top \Delta_n V_n\|_2.
\end{aligned}    
\end{equation*}
By Lemma \ref{lem:max_eig0},  $\|\alpha_n\|_2$ and $n^{-1/2} \|f(Z_n)\|_2$ are bounded almost surely eventually. For $z \in \mathbb{C}^+$, using the identity $(zI - A)^{-1} - (zI - B)^{-1} = (zI - A)^{-1} (A - B)(z I - B)^{-1}$ and $\|(zI-A)^{-1}\|_2 \leq \Im(z)^{-1}$,
\begin{equation}
\begin{aligned} \label{g4y}
  \Big\|G_n(z) - \Big( zI_{p-r}- \frac{1}{n} W_n^\top f(Z_n)^\top f(Z_n) W_n\Big)^{-1} \Big\|_2 \leq &  \frac{1}{\Im(z)^2} \Big\| W_n^\top \Big(Y_n^\top Y_n - \frac{1}{n} f(Z_n)^\top f(Z_n)\Big) W_n \Big\|_2 \\ 
  \leq & \frac{\|\Delta_n\|_2}{\Im(z)^2} \Big (\frac{2}{\sqrt{n}}\|f(Z_n)\|_2 + \|\Delta_n\|_2 \Big) .  
\end{aligned}    
\end{equation}
As $\|\Delta_n\|_2 \xrightarrow{a.s.}0$ by  Theorem \ref{thrm1}, we obtain  
\begin{align}
    &&  \hspace{2cm} K_n(z) - \overline K_n(z)  \xrightarrow{a.s.} 0 , && \hspace{2cm} z \in \mathbb{C}^+ \label{z7c}
\end{align}
(as $K_n$ is of size $r \times r$, operator norm and elementwise convergence are equivalent). 

Lemma \ref{lema5} and (\ref{a17s}) imply that $\overline K_n(z)$ converges to a diagonal matrix: 
\begin{align} & \hspace{3.9cm}
  \overline K_n(z) \xrightarrow{a.s.} (-z  \gamma m_\gamma(z) + 1 - \gamma) (I_r + \Lambda^2) ,  &  \hspace{1.0cm} \,\,\,\,\,\,\,\,\,\,\, z\in \mathbb{C}^+. \label{z8c}
\end{align}
Moreover, by the Arzela--Ascoli theorem, the convergence in (\ref{z7c}) and (\ref{z8c}) is uniform on compact subsets of $\mathbb{C}^+$. 
Since uniform convergence of an analytic sequence implies uniform convergence of the derivative,
\begin{align} \label{pxg}
&  \hspace{3.6cm} \partial_z K_n(z) \xrightarrow{a.s.}   -\big( \gamma m_\gamma (z) +  z \gamma \partial_z  m_\gamma (z) \big) ( I_r + \Lambda^2) , & \hspace{.78cm} \,\, z \in \mathbb{C}^+,
\end{align}
the convergence occurring uniformly on compact subsets of $\mathbb{C}^+$.
In particular, since $\sigma_i^2(Y_n) \xrightarrow{a.s.} \lambda_+ \coloneqq (1+\sqrt{\gamma})^2$ by (\ref{g27v}), we have for $\eta > 0$ that
\begin{align}
    \partial_z  K_n(\sigma_i^2(Y_n) + i \eta) \xrightarrow{a.s.} -\big(\gamma m_\gamma (\lambda_+ + i \eta) + (\lambda_+ + i \eta) \gamma \partial_z m_\gamma(\lambda_+ + i \eta) \big) (I_r + \Lambda^2) .  \label{z9c}
\end{align}
As $\eta > 0$ is arbitrary, (\ref{z6c}) follows from (\ref{z9c}), the bound
\[
 \sigma_r\big(\partial_z K_n(\sigma_i^2(Y_n))\big) \geq 
 \sigma_r\big(\partial_z K_n(\sigma_i^2(Y_n)+ i \eta)\big) ,  
\]
and the fact that $\big| \gamma m_\gamma (\lambda_+ + i \eta) + (\lambda_+ + i \eta) \gamma \partial_z m_\gamma(\lambda_+ + i \eta) \big| \rightarrow \infty$ as $\eta \rightarrow 0$. 

For supercritical singular values, $\tau(f,\mu) X_n > \gamma^{1/4}$, let $\lambda_i \coloneqq \lambda(\tau(f, \mu) \sigma_i(X_n), \gamma)$. Since $\lambda_i > \lambda_+$ by (\ref{g27v}), arguments similar to (\ref{g4y})--(\ref{z9c}) yield 
\begin{align}
    \partial_z  K_n(\sigma_i^2(Y_n)) \xrightarrow{a.s.} - \big(\gamma m_\gamma (\lambda_i) + \lambda_i \gamma \partial_z m_\gamma(\lambda_i) \big) (I_r + \Lambda^2) .  
\end{align}
As the elements of $\Lambda$ are distinct and in decreasing order,  the Davis-Kahan theorem (Theorem 2 of \cite{YWS14}) implies $V^\top v_i(Y_n)$ (which satisfies  $K_n(\sigma_i^2(Y_n)) V_n^\top v_i(Y_n) = \sigma_i^2(Y_n) V_n^\top v_i(Y_n)$) converges to the $i$-th standard basis vector (the length-$r$ vector with one as the $i$-th coordinate and zeros elsewhere). Using (\ref{b6s}) and the identity 
\[
\big( 1 - \big(\gamma m_\gamma (\lambda_i) + \lambda_i \gamma \partial_z m_\gamma(\lambda_i) \big) ( 1 +  \tau(f,\mu)^2 \sigma_i^2(X_n)) \big)^{-1} = c_2^2(\lambda_i, \gamma) , 
\]
we obtain $\langle v_i(X_n), v_j(Y_n) \rangle^2 \xrightarrow[]{a.s.} \delta_{ij} \cdot c_2^2(\lambda_i, \gamma)$.  The proof of the corresponding result for left singular vectors is similar and omitted. 

\end{proof}

To avoid calculations involving the logistic distribution and its associated orthogonal polynomials, notice that (\ref{plmokn2}) has an equivalent representation: taking  $z_{ij} \stackrel{i.i.d.}{\sim} \mathcal{N}(0,1)$,
\begin{align} \label{g4ys}
    {\bf 1}\big(\tilde x_{ij} + z_{ij} \leq 0 \big) \sim \mathrm{Ber}(\mathrm{logistic}(x_{ij})), 
\end{align}
where $\tilde x_{ij} \coloneqq -\Phi^{-1}(\mathrm{logistic}(x_{ij}))$ and $\Phi$ denotes the standard Gaussian cumulative distribution function. 

\begin{proof}[Proof of Corollary \ref{thrm_bin}]
In view of (\ref{g4ys}), we decompose (\ref{f4q}) as $Y_n \stackrel{d}{=} \sum_{k=1}^m Y_{n}^{(k)}$, where 
\begin{align*} 
 Y_{n}^{(k)}  &\coloneqq  \frac{1}{\sqrt{n}} f(\sqrt{n}\widetilde X_n + Z_{n}^{(k)}),   \\
  f(z) &\coloneqq {\bf 1}(z \leq 0) - \frac{1}{2} ,  \\
 \widetilde X_n &\coloneqq - \frac{1}{\sqrt{n}}\Phi^{-1}\Big(\mathrm{logistic}\Big(\sqrt{\frac{n}{m}}X_n\Big)\Big),  
\end{align*}
and $Z_{n}^{(k)}$ are independent matrices with i.i.d.\ Gaussian elements. As the  elements of $\sqrt{n/m} X_n$ uniformly converge to zero, we have
\[
\widetilde X_n= (1+o(1)) \sqrt{\frac{\pi}{8m}} X_n ;
\]
we therefore define $A_n \coloneqq \sum_{k=1}^m A_n^{(k)}$ where 
\[A_{n}^{(k)} \coloneqq a_1 \sqrt{\frac{\pi}{8m}} X_n + \frac{1}{\sqrt{n}} f(Z_{n}^{(k)}) =  \frac{1}{4\sqrt{m}} X_n + \frac{1}{\sqrt{n}} f(Z_{n}^{(k)}) .\]

 Given Corollary \ref{cor1}, it suffices to prove $ m^{-1/2}  \|Y_n - A_n \|_2 \xrightarrow{a.s.} 0$. For fixed $m$, this is an immediate consequence of
Theorem \ref{thrm1}, provided $\sqrt{n/m}X_n$ satisfies assumption (i). To accommodate $m \rightarrow \infty$ and the weaker condition Corollary \ref{thrm_bin} imposes on $X_n$, we
directly calculate moments of the elements of $Y_n - A_n$: 
\begin{equation}
\begin{aligned}
     \E \big(y_{ij}^{(k)} - a_{ij}^{(k)}\big) &=  -\frac{x_{ij}}{4\sqrt{m}} + \frac{1}{\sqrt{n}} \int_{-\sqrt{n} \tilde x_{ij}}^0 \phi(z) dz = -\frac{x_{ij}}{4\sqrt{m}} + \frac{1}{\sqrt{n}} \Phi(-\sqrt{n} \tilde x_{ij}) - \frac{1}{2\sqrt{n}} \\
     & =  \frac{n x_{ij}^3}{48 m^{3/2}} + O\Big(\frac{n^2 x_{ij}^5}{m^{5/2}}\Big),\\
     \E \big(y_{ij}^{(k)} - a_{ij}^{(k)}\big)^4 &=  \int_{\mathbb{R}\backslash[-\sqrt{n}\tilde x_{ij},0]} \Big(\frac{x_{ij}}{4\sqrt{m}}\big)^4 \phi(z) dz + \int_{-\sqrt{n}\tilde x_{ij}}^0 \Big(\frac{x_{ij}}{4\sqrt{m}} + \frac{1}{\sqrt{n}}\Big)^4 \phi(z) dz \\ & = \frac{x_{ij}}{4 n^{3/2} \sqrt{m}} + O\Big( \frac{x_{ij}^2}{m n}\Big) . \label{a30hut}
\end{aligned}    
\end{equation}
Now, using (\ref{a30hut}) and the uniform convergence to zero of the elements of $\sqrt{n/m}X_n$, 
\begin{equation}
    \begin{aligned}
       \sqrt{m} \big\| \E \big( Y_n^{(k)} - A_n^{(k)}\big) \big\|_2  & \leq   \sqrt{m} \big\| \E \big( Y_n^{(k)} - A_n^{(k)}\big) \big\|_F  \leq \sqrt{m} \bigg( \sum_{i=1}^n \sum_{j=1}^p \big [ \E \big( y_{ij}^{(k)} - a_{ij}^{(k)} \big) \big]^2 \bigg)^{1/2}  \\
       & \leq \|X_n\|_F \cdot \sup_{\substack{1 \leq i \leq n, 1 \leq j \leq p, \\ x_{ij} \neq 0}} \Big|\frac{n x_{ij}^2}{48 m^{2}}\Big| + o(1)  = o(1) .
    \end{aligned}
\end{equation}
As this holds uniformly in $k \in \{1, \ldots, m\}$, the operator norm of the $m^{-1/2} \E(Y_n-A_n)$ is negligible: 
\begin{align}
         \frac{1}{\sqrt{m}} \| \E(Y_n - A_n) \|_2 & \leq \frac{1}{\sqrt{m}} \sum_{k=1}^m \big\| \E \big( Y_n^{(k)} - A_n^{(k)}\big) \big\|_2 = o(1) .
\end{align}
Furthermore, since $Y_n^{(1)} - A_n^{(1)}, \ldots, Y_n^{(m)} - A_n^{(m)}$ are independent, 
\begin{equation}
\begin{aligned}
    \E |y_{ij} - a_{ij}|^4 & =   \sum_{k, k', \ell, \ell' = 1}^m \E \Big[ \big(y_{ij}^{(k)} - a_{ij}^{(k)}\big) \big(y_{ij}^{(k')} - a_{ij}^{(k')}\big)\big(y_{ij}^{(\ell)} - a_{ij}^{(\ell)}\big) (y_{ij}^{(\ell')} - a_{ij}^{(\ell')}\big) \Big] \\
    &= \frac{\sqrt{m}x_{ij}}{4n^{3/2}}  +\frac{3(m-1) x_{ij}^2}{16n} + O\Big(\frac{x_{ij}^2}{n}\Big)  .\label{a35} 
\end{aligned}
  \end{equation}
To obtain the last equality, we used that the sum is dominated by terms in which the four indices are equal (there are $m$ such terms) or paired (there are $3m(m-1)$ such terms).

Thus, denoting $W_n \coloneqq \sqrt{n/m}(Y_n-A_n - \E (Y_n - A_n))$, the elements of $W_n$ are mean zero and have  
fourth moments uniformly converging to zero: 
\begin{align}
\lim_{n \rightarrow \infty}\sup_{1 \leq i \leq n, 1 \leq j \leq p} \E |w_{ij}|^4 = 0 .
\end{align}
Applying Lemma  \ref{lem:max_eig} to $W_n$ therefore completes the proof: 
\begin{align}
    \frac{1}{\sqrt{m}} \|Y_n - A_n\|_2 \leq \frac{1}{\sqrt{n}} \|W_n\|_2  + \frac{1}{\sqrt{m}}\|\E(Y_n-A_n)\|_2  \xrightarrow[]{a.s.} 0 .
\end{align}
\end{proof}

\begin{proof}[Proof of Corollary \ref{cor3}] The first portion of the corollary is immediate; we therefore assume $\mu$ has finite moments, differentiable density $\omega$, $\text{supp}(\mu) = \mathbb{R}$, and $\omega'/\omega$ is polynomially bounded. 
Using integration by parts,
\begin{align}
    a_k = \int_{-\infty}^\infty  f^*(z) q_k(z) \omega(z) dz =-\int_{-\infty}^\infty  q_k(z) \omega'(z) dz = \int_{-\infty}^\infty q_k'(z) \omega(z) dz  = b_k.
\end{align}
Thus, by (\ref{assumption4.0}),
\begin{align}
    \tau(f^*,\mu) = \frac{1}{ \|f^*\|_\mu} \sum_{k=1}^\infty a_k^2 = \|f^*\|_\mu = \bigg(\int_{-\infty}^\infty \frac{(\omega'(z))^2}{\omega(z)} dz \bigg)^{1/2}.
    \end{align}
We refer to Proposition 4.2 of \cite{Perry} for proof that $\mathcal{I}(\omega) \geq 1$ and that the inequality is strict for all non-Gaussian distributions. 
\end{proof}

\begin{proof}[Proof of Corollary \ref{cor:trunc}]
    For $x\in[0,2c]$, we have
\begin{equation}
\begin{aligned} \label{truncato1}
        \int (f_c(x+z) - f_c(z)) d\mu(z) & =  \int_{-c-x}^{-c} (x+z) \omega(z) dz  + \int_{-c}^{c-x} x \omega(z) dz -\int_{c-x}^c z \omega(z) dz\\
        &= x (F_\mu(c-x)- F_\mu(-c-x)) + \int_{-c-x}^{-c} z \omega(z) dz - \int_{c-x}^{c} z \omega(z) dz ,
\end{aligned}
\end{equation}
and for $x < [-2c,0)$, 
\begin{equation}
\begin{aligned} \label{truncato2}
   \hspace{.42cm} \int (f_c(x+z) - f_c(z)) d\mu(z) & =  -\int_{-c}^{-c-x} z \omega(z) dz  + \int_{-c-x}^{c} x \omega(z) dz +\int_{c}^{c-x} (x+z) \omega(z) dz\\
        &= x (F_\mu(c-x)- F_\mu(-c-x)) - \int_{-c}^{-c-x} z \omega(z) dz + \int_{c}^{c-x} z \omega(z) dz . 
\end{aligned}
\end{equation}
As the elements of $\sqrt{n} X_n$ uniformly converge to zero and $\omega$ and $F_\mu$ are continuous at $\pm c$,  
\begin{align} \label{a38}
    \int (f_c(\sqrt{n}x_{ij} +z) - f_c(z)) d\mu(z) = \sqrt{n} x_{ij} [ F_\mu(c)-F_\mu(-c) -c (\omega(c) + \omega(-c)) ] + o(\sqrt{n}x_{ij}) ,
\end{align}
implying $\|\E (Y_n - A_n)\|_2 \rightarrow 0$ analogously to (\ref{as9jd}) and (\ref{qerg}). Moreover, calculations similar to (\ref{truncato1}) and (\ref{truncato2}) yield that the second and fourth moments of $\sqrt{n}(Y_n - A_n)$ are uniformly bounded:
\begin{align*}
& \lim_{n \rightarrow \infty} \sup_{1 \leq i \leq n, 1 \leq j \leq p} n \E(y_{ij} - a_{ij})^2 = 0 , &     \limsup_{n \rightarrow \infty} \sup_{1 \leq i \leq n, 1 \leq j \leq p} n^2 \E(y_{ij} - a_{ij})^4 < \infty .
\end{align*}
Thus, as in the proof of Lemma \ref{lem1}, we conclude using Lemma \ref{lem:max_eig} that $   \|Y_n - A_n\|_2 \xrightarrow[]{a.s.} 0$.

\end{proof}

\section*{Acknowledgements}

The author is grateful to Apratim Dey, David Donoho, Elad Romanov, and Tselil Schramm for discussions and comments.

\end{document}